\documentclass[10pt,a4paper]{article}
\usepackage[margin=1in]{geometry}
\usepackage[latin1]{inputenc}
\usepackage[T1]{fontenc}
\usepackage{amsmath,amsthm}
\usepackage{amsfonts}
\usepackage{amssymb}
\usepackage{graphicx}
\usepackage{bbold,color}
\usepackage{enumerate}
\usepackage[symbol,perpage]{footmisc}
\newtheorem{theorem}{Theorem}
\newtheorem{proposition}[theorem]{Proposition}
\newtheorem{lemma}[theorem]{Lemma}
\newtheorem{corollary}[theorem]{Corollary}
\numberwithin{theorem}{section}
\newcommand{\T}{\mathbb{T}}
\newcommand{\R}{\mathbb{R}}
\newcommand{\Z}{\mathbb{Z}}
\newcommand{\C}{\mathbb{C}}
\newcommand{\N}{\mathbb{N}}

\newcommand{\K}{\mathcal{K}}
\renewcommand{\L}{\mathcal{L}}

\newcommand{\V}{\mathcal{V}}
\newcommand{\U}{\mathcal{U}}
\newcommand{\bV}{{\bar{\mathcal{V}}}}
\newcommand{\bU}{{\bar{\mathcal{U}}}}
\newcommand{\bc}{{\bar {c}}}
\newcommand{\M}{\mathcal{M}}
\newcommand{\D}{\mathcal{D}}

\newcommand{\Hd}{\mathcal{H}}
\newcommand{\Ad}{\mathcal{A}}
\newcommand{\F}{\mathcal{F}}
\renewcommand{\P}{\mathcal{P}}

\newcommand{\md}{h}

\newcommand{\dd}{\mathrm{d}}
\newcommand{\half}{\tfrac{1}{2}}

\DeclareMathOperator{\Leb}{Leb}
\DeclareMathOperator{\esssup}{ess-sup}
\DeclareMathOperator{\spn}{span}
\DeclareMathOperator{\sign}{sign}
\DeclareMathOperator{\sech}{sech}
\DeclareMathOperator{\im}{im}

\def\XXint#1#2#3{{\setbox0=\hbox{$#1{#2#3}{\int}$}
		\vcenter{\hbox{$#2#3$}}\kern-.5\wd0}}

\newcommand{\rev}[1]{#1}

\author{Caroline Wormell}
\title{Orthogonal polynomial approximation and Extended Dynamic Mode Decomposition in chaos}
\begin{document}
	\maketitle

	\begin{abstract}
Extended Dynamic Mode Decomposition (EDMD) is a data-driven tool for forecasting and model reduction of dynamics, which has been extensively taken up in the physical sciences. While the method is conceptually simple, in deterministic chaos it is unclear what its properties are or even what it converges to. In particular, it is not clear how EDMD's least-squares approximation treats the classes of differentiable functions on which chaotic systems act.

We develop for the first time a general, rigorous theory of EDMD on the simplest examples of chaotic maps: analytic expanding maps of the circle. To do this, we prove a new, basic approximation result in the theory of orthogonal polynomials on the unit circle (OPUC) and apply methods from transfer operator theory. We show that in the infinite-data limit, the least-squares projection error is exponentially small for trigonometric polynomial observable dictionaries. As a result, we show that forecasts and Koopman spectral data produced using EDMD in this setting converge to the physically meaningful limits, exponentially fast with respect to the size of the dictionary.
This demonstrates that with only a relatively small polynomial dictionary, EDMD can be very effective, even when the sampling measure is not uniform. Furthermore, our OPUC result suggests that data-based least-squares projection may be a very effective approximation strategy more generally.
	\end{abstract}

\section{Introduction}
%
%

Nonlinear systems with complex dynamics are found across the physical and human sciences, counting among them climate systems, fluid flows and economic systems \cite{Fuchs14,Brock18}. Unlike linear systems, they rarely admit useful analytic solutions or have much obvious mathematical structure that can be used to make sense of them. Furthermore, the information we may have on these systems can often be very limited: for example, we may only have access to some empirical observations of its evolution. General methods of studying nonlinear systems in terms of well-understood mathematical objects, that can be performed using empirical data, are therefore an important tool in scientific endeavour.

One approach to do this is to use a surrogate linear system to study the nonlinear system. For example, one can study functions (``observables'') on the phase space, and construct a linear {\it Koopman operator} which maps observables to their forward expectations under the flow \cite{Budivsic12}. Computationally, only a finite-dimensional vector space of observables is considered: the Koopman operator may not preserve this space of observables, but can be projected onto the finite-dimensional space, with good results if the observables are well-chosen. This projection can be done by least-squares on the empirical data, making it widely applicable. The computational representation of the Koopman operator can then be studied in terms of its spectrum and eignefunctions, which describe dynamical properties such mixing rates, almost-invariant sets in phase space \cite{Froyland14, Colbrook21, Colbrook23}.

Many different numerical methods, usually described as Dynamic Mode Decomposition (DMD) variants, employ this approach: for example, the classic DMD uses linear functions of delay coordinates as its observables. The aim of Extended Dynamic Mode Decomposition (EDMD), however, is to choose a large ``dictionary'' of observables that can be effectively used to approximate any function on the phase space \cite{Williams15}: commonly, polynomials rather than simply linear functions are used. EDMD and its variants have recently been used to great success in a broad variety of settings, including in airflow, molecular dynamics, industrial processes, and decision networks \cite{Mauroy20, Colbrook22, Colbrook23, Son22}. 

Given its applicability, there has been a lot of interest in the mathematical theory behind EDMD, and in particular the interpretation of its spectral data. This theory has largely been explored by considering the Koopman operator $\mathcal{K}\varphi = \varphi \circ f$ of the dynamics $f$ as acting on $L^2(\mu)$, where $\mu$ is the (often dynamically invariant) sampling measure of the empirical data points: the isometric nature of the Koopman operator on this space is exploited to obtain various spectral convergence results regarding the $L^2(\mu)$ spectrum, which lies on the unit circle for invertible dynamical systems, and in some non-invertible settings covers the unit \rev{disc} \cite{Keller89,Slipantschuk20}. However, the actual spectra of Koopman matrices obtained using EDMD bear much closer resemblance to the spectra of quasi-compact Markov operators \cite{Hennion01}, whose spectrum is mostly located inside the unit circle (see Figure~\ref{fig:circlemap} for an example).

The explanation for this can be found in the fact that the Koopman operator is the dual of the transfer operator \cite{Slipantschuk20,Bandtlow23}, a typically quasi-compact Markov operator that tracks the movement of probability densities under the dynamics, up to some reweighting. Transfer operators have been heavily studied in dynamical systems, and have a strong and well-developed functional-analytic theory, including numerical approximation of transfer operators \cite{Keller99, Dellnitz02, Wormell19, Bandtlow20}. Nevertheless, with some exceptions \cite{Froyland14, Slipantschuk20,Bandtlow23}), cross-pollination between the separate transfer operator and Koopman operator communities has been regrettably sparse.

Among the transfer operator discretisations, Extended DMD bears close resemblance to Fourier methods for transfer operators \cite{Wormell19, Crimmins20}, which make a standard orthogonal projection onto trigonometric polynomals with respect to Lebesgue measure. 
Proving the convergence of discretisations with respect to these projections can be done by using the orthogonality of derivatives of the basis functions, a property which has been used to prove convergence of EDMD in the instance when the sampling measure $\mu$ is uniformly distributed (i.e. exactly Lebesgue measure) \cite{Slipantschuk20}. It is very rare, however, for a sampling measure to have this property: in fact, only a few, famous, families of measures do \cite{Krall36}. For more general sampling measures $\mu$, new approximation theory must be developed to carry over transfer operator results to DMD variants: this is the advance of this paper.


\subsection{Polynomial approximation result}

We do this by proving a basic and general result in orthogonal polynomial approximation (Theorem~\ref{t:GalerkinGood}), which to the best of our knowledge is new. We consider the operator $\mathcal{P}_K$  which accepts functions on the periodic interval $\T/2\pi\Z$ and outputs their $L^2(\mu)$-orthogonal projection onto trigonometric polynomials of degree $\leq K$, i.e. into the space
\[ \spn \{1, e^{-ix}, e^{ix}, e^{-2ix}, e^{2ix}, \ldots, e^{-i(K-1)x}, e^{i(K-1)x} \}. \]
In the special case where $\mu$ is uniform measure, $\P_K$ is just the Dirichlet kernel operator $\D_K$ (which truncates Fourier modes in the standard way):
\[ (\D_K \psi)(x) = \sum_{k=1-K}^{K-1} \frac{1}{2\pi} \int_0^{2\pi} e^{ik(x-y)} \psi(y)\,\dd y  \]

Our result is that, under some fairly weak stipulations on the density of $\mu$, the $L^2(\mu)$ projection $\P_K$ is about as good an approximation as the Dirichlet kernel (i.e. truncating Fourier modes), where the metric for ``good approximation'' is the operator error between Sobolev-style Hilbert spaces.

To prove this result we bring in ideas from the theory of orthogonal polynomials on the unit circle (OPUC) \cite{Simon05}, which can be used to characterise bases of trigonometric polynomials in which the least-squares projection $\mathcal{P}_K$ is diagonal. The ideas in this body of theory allow one to relate this basis quite effectively to the usual monomial basis. 

Aiming to be as general as possible, we study the Hilbert spaces $W^\sigma := \mathcal{F}^{-1}[\sigma \ell^2(\Z)]$ weighted by so-called {\it Beurling weights}: even functions $\sigma:\Z \to \R^+$ that are non-decreasing on the natural numbers, and obey $\sigma(j)\sigma(k) \leq \sigma(|j|+|k|)$. Among these spaces include the usual integer and fractional Sobolev Hilbert spaces (up to norm equivalence), as well the Hardy-Hilbert spaces $\Hd^2_t$ we used to study EDMD (see Section~\ref{ss:MainResult}); when $\sigma^{-1}$ is square-summable, our spaces $W^\sigma$ are also examples of reproducing kernel Hilbert spaces \cite{Aronszajn50}.
\begin{theorem}\label{t:GalerkinGood}
	Suppose that $\sigma, \tau: \Z \to \R_+$ are Beurling weights with $\tau/\sigma$ decreasing on $\N$.
	
	Suppose furthermore that $\mu = \md \,\dd x$ is a positive measure on $\T$ with $M^{-1} \leq \md(x) \leq M$ on $\T$ and $\|\sigma \mathcal{F}[(\log \md)']\|_{\ell^q} \leq A$ for some $q<\infty$.
	
	Then there is a constant $C_{\P}$ increasing in and dependent only on $q, A, M$ such that
	\[ \|I - \mathcal{P}_K \|_{W^\sigma \to W^\tau} \leq C_{\P} \|I - \D_K\|_{W^\sigma \to W^\tau} = C_{\P} \frac{\tau(K)}{\sigma(K)}. \]
	where $\mathcal{P}_K$ is the $L^2(\mu)$-orthogonal projection onto trigonometric polynomials of degree less than $K$, and $\mathcal{D}_K$ is the Dirichlet kernel.
	%
\end{theorem}

This theorem can be expected to generalise to higher dimensions, retaining for analytic sample measures an $\mathcal{O}(e^{-cK})$ convergence rate for a dictionary of degree $< K$ trigonometric polynomials.

Our result suggests that if $\mu$ is smooth then approximation of smooth functions by polynomials in $L^2(\mu)$ is very powerful. As a result, in EDMD and beyond, we should be eager to make use of this kind of polynomial approximation where it arises, for example in least squares approximation from data.

\subsection{The dynamics picture}
The goal of this paper is to obtain some a priori knowledge about how EDMD performs on chaotic systems. However, as a deterministic chaotic system becomes more structurally complex, its dynamical systems theory quickly becomes very technically involved \cite{Blumenthal17}. For this reason, we choose as an initial example uniformly expanding maps of the torus $\T := \R / 2\pi \Z$: that is, maps $f: \T \to \T$ such that for all $\theta \in \T$, $\inf |f'(\theta)|\geq \gamma > 1$. An example of such a map is given in the top left of Figure~\ref{fig:circlemap}. These are common first examples for understanding phenomena in chaotic dynamics \cite{Baladi00,Wormell19}. We note that they are sometimes studied as self-maps of the complex unit circle in coordinates $z = e^{i\theta}$. Some additional smoothness is required to obtain good results: for simplicity, we will assume that $f$ is analytic. We will also assume that $\mu$ has analytic density, a reasonable assumption as the physical invariant measures of such maps have analytic densities: nevertheless our fundamental polynomial approximation result (Theorem~\ref{t:GalerkinGood}) applies to any regularity.

The key classical dynamical systems result about these maps is this. Given observable functions in $C^\infty(\T)$ and $\psi \in L^1(\T)$ \rev{(or even in some larger space of hyper-distributions)}, we can expand their lag correlations as a sum of exponentially decaying functions:
\begin{equation} \int_\T \varphi\, \psi\circ f^n\, \dd \mu \sim \sum_{j=0}^\infty \sum_{m=1}^{M_j}   \alpha_{j}^{(m)}(\varphi) \beta_{j}^{(m)} (\psi)\, n^m \lambda_j^n \label{eq:RuellePollicott} \end{equation}
where $\alpha_{j}^{(m)} (\varphi) := \int a_{j}^{(m)} \varphi\,\dd\mu$ for some function $a_{j}^{(m)} \in \Hd^2_t$, and $\beta_{j}^{(m)} : \Hd^2_t \to \R$ is a bounded functional that can be written formally as $\beta_{j}^{(m)}(\psi) = \int \psi b_j^{(m)},\dd x$, for some generalised hyperdistribution $b_j^{(m)}$ in the space $\Hd^2_{-t}$ (defined in Section~\ref{s:Results}). The multiplicities $M_j$ are generically $1$, in which case we drop the superscripts. 

The complex $\lambda_j$, which have modulus no greater than one, are known in the dynamical systems literature as Ruelle-Pollicott resonances: they are a kind of canonical, function-space independent eigenvalue of a transfer operator of the system \cite{Baladi18}. In particular, they are canonical eigenvalues of a transfer operator that tracks movement of probability density weighted with respect to sampling measure $\mu$: when $\mu$ is the physical measure (i.e. long-term ergodic distribution) of the system, this is the Perron-Frobenius operator. 

These resonances and associated linear operators determine many important properties of the dynamical system. For example, for generalised eigenfunctions $a_{j}^{(m)}$ of an eigenfunction $\lambda_j$ close to $1$, the sets $\{x: a_j(x)>0\}$ denote almost-invariant sets with respect to the dynamics, as equivalently do regions where the generalised eigendistributions $\beta_j$ are positive on average \cite{Froyland03,Froyland14}. 

\rev{Mixing results such as \eqref{eq:RuellePollicott} are proven by showing appropriate compactness properties on a sensible Banach space for the Koopman operator, or more commonly its adjoint in $L^2(\dd x)$, the transfer operator. This approach also governs numerical study \cite{Keller99}. We will define in Section~\ref{s:Results} a certain family of Hilbert spaces of analytic functions  $\{\Hd^2_t\}_{t>0}$, with $\Hd^2_s \subset \Hd^2_t \subset C^\infty(\T,\C)$ for $s>t>0$. For an analytic circle map the transfer operator is bounded and compact on $\Hd^2_t$ for small enough $t$ \cite{Bandtlow17}. Let $\Hd^2_{-t} \supset L^2(\T,\C)$ be spaces of generalised distributions that we identify with the dual space of $\Hd^2_t$. The Koopman operator is therefore bounded and compact on $\Hd^2_{-t}$.}
 
\begin{figure}[t]
	\centering
	\includegraphics[width=\linewidth]{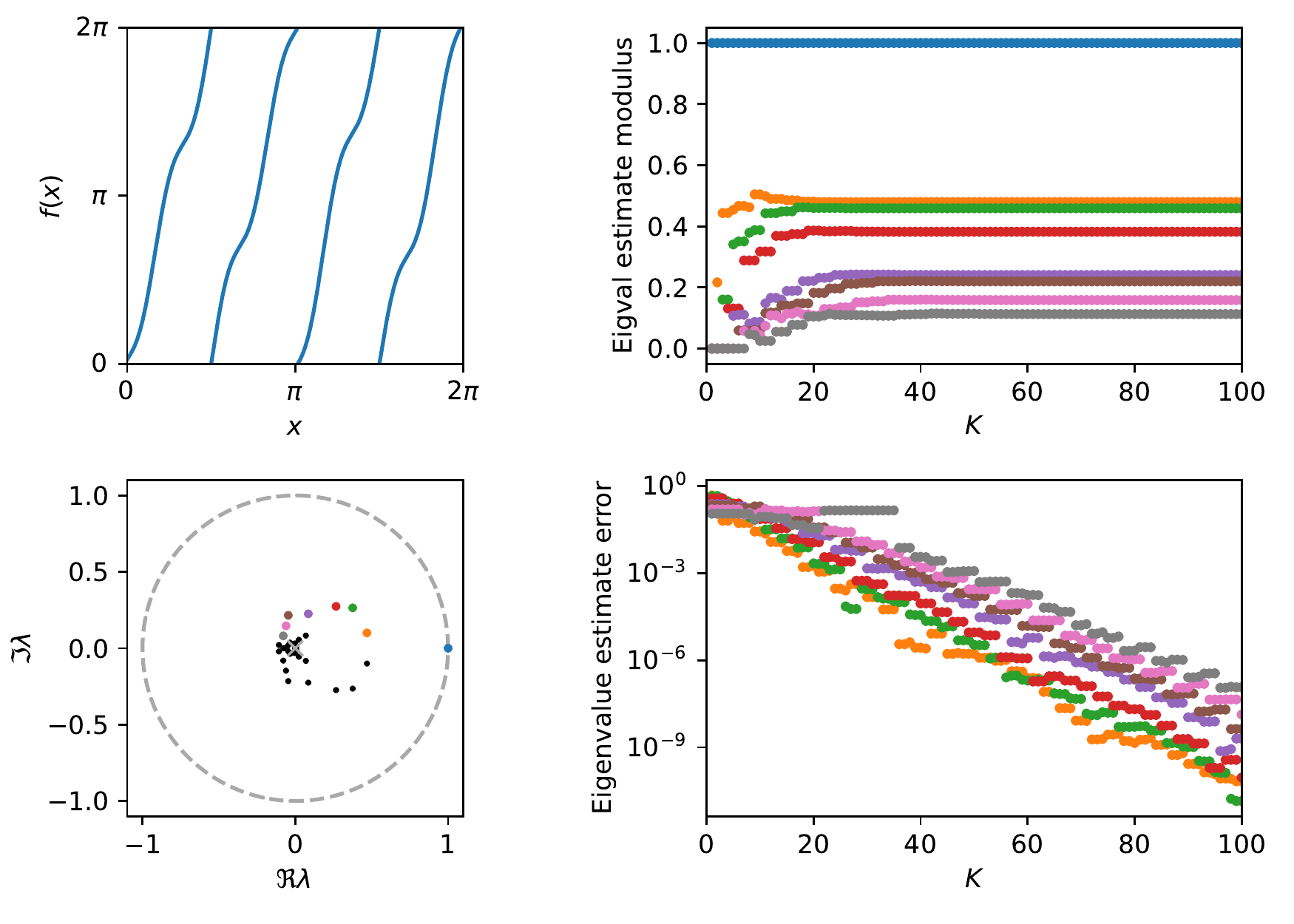}
	\caption{Top left: graph of the circle map $f(x) =  4x - 0.4\sin 6x+0.08\cos 3x \mod 2\pi$. Bottom left: Ruelle-Pollicott resonances (i.e. Perron-Frobenius operator eigenvalues) of $f$ with certain eigenvalues marked in colour. Top right: modulus of EDMD eigenvalues obtained with $\mu$ the physical measure of $f$, with colours corresponding to eigenvalues in the bottom left. Bottom right: exponential convergence with respect to dictionary size of the errors between EDMD eigenvalues and Ruelle-Pollicott responances, with corresponding colours. Ruelle-Pollicott resonances estimated using a Fourier transfer operator discretisation \cite{Wormell19}.}
	\label{fig:circlemap}
\end{figure}

\subsection{Main result on EDMD}\label{ss:MainResult}
The EDMD algorithm makes a best operator approximation of the Koopman operator $\K$ within the span of the observable dictionary $\{\psi_k \}_{|k|< K}$. This best approximation done with respect to the metric of $L^2(\mu_N)$, where $\mu_N$ is the empirical measure of the data points $\{x_n\}_{n=1,\ldots,N}$. This least-squares approximation of $\K$ can be computed in the $\psi_k$ basis as the Koopman matrix
$$ (\Psi_{(0)}^* \Psi_{(0)})^{-1} (\Psi_{(0)}^* \Psi_{(1)}) $$
where $(\Psi_{(0)})_{nk} = \psi_k(x_n)$ and $(\Psi_{(1)})_{nk} = \K\psi_k(x_n) = \psi_k(f(x_n))$. The key benefit of EDMD over other algorithms is that these matrices can be filled directly from observations of the system. 

As the number of points $N$ goes to infinity, this operator converges to a (still finite-dimensional) continuum limit operator $\K_K \rev{:= \P_K\K |_{\im \P_K}}$, which is the least-squares best approximation in $L^2(\mu)$ of $\K$ within the dictionary of observables $\im \P_K= \spn\{\psi_k\}_{|k|< K}$ \cite{Klus16}. \rev{We will extend this operator naturally to act on $\Hd^2_{-t}$ as $\K_K= \P_K \K \P_K$.} As a result of this convergence, the Koopman matrix's spectral data also converge to those of $\K_K$ (see Figure~\ref{fig:dataeig}).

The main result of this paper is as follows. Let $\lambda_{j,K}$ be the eigenvalues of the EDMD estimate of the Koopman operator $\K_K$ obtained using the function dictionary 
\begin{equation}  \{ e^{-i(K-1)x}, e^{-i(K-2)x}, \ldots, e^{-ix}, 1, e^{ix}, \ldots, e^{i(K-1)x}\},\label{eq:ObsDict}\end{equation}
in the limit of infinite data points. Each of these eigenvalues, which will be simple for $K$ large enough if the corresponding $\lambda_j$ is simple, have \rev{left eigenvectors $a_{j,K}$ and right eigenvectors $b_{j,K}$ that are polynomials}, which we can consider in the observable basis.

\begin{theorem}\label{t:SimpleSpectral}
	Suppose $f$ is analytic and $\mu$ has analytic density. Then for some $c, h, \zeta >0$ and $\kappa \in (0,1)$ depending only on $f, \mu$, the following properties hold:
	
	\begin{itemize}
		\item The Hausdorff distance between the spectrum of the Koopman matrix $\K_K$ and the Ruelle-Pollicott responances $\{\lambda_j\}_{j\in \N}$ is $\mathcal{O}(e^{-c\sqrt{K}})$ as $K \to \infty$.
		\item 	Let $\lambda_j$ be any resonance with multiplicity $1$ (see Theorem~\ref{t:Deterministic} for a more general result). Then there are \rev{left (resp. right) eigenvectors $a_{j,K}$ (resp. $b_{j,K}$) of $\K_K$} such that for $t \in [0,\zeta]$,
		$$ |\lambda_{j,K} - \lambda_j|, \|a_{j,K} - a_j\|_{\Hd^2_{\kappa t}}, \|b_{j,K} - b_j\|_{\Hd^2_{-t}} = \mathcal{O}(e^{-h t K}).$$
	\end{itemize} 
\end{theorem}

An illustration of this theorem is given in Figure~\ref{fig:circlemap}, with eigenvalue estimates converging exponentially with the dictionary size $K$ as predicted: note that the leading constants generally increase as the modulus of the eigenvalue decreases.

The consequence of Theorem~\ref{t:SimpleSpectral} is that, in this setting, EDMD gives very accurate estimates about the system's rates of mixing, encoded through the Ruelle-Pollicott resonances. It also means that the so-called {\it DMD modes} (or left Koopman modes) $a_{j,K}$ and {\it (right) Koopman modes} $b_{j,K}$ obtained respectively as left and right eigenvectors of the Koopman operator approximation closely approximate certain limiting objects that give information about the long-term dynamical properties of the system. For $|\lambda_j|$ large, the limiting DMD mode $a_j$ and right Koopman mode $b_j$ encode structures that persist under the dynamics on $\mathcal{O}(\log |\lambda_j|)$ timescales, including almost-invariant sets \cite{Froyland14, Dellnitz02}; for smaller eigenvalues they give information about finer mixing properties. Notwithstanding that the maps we consider are extremely structurally simple and have transfer operators with particularly nice approximation properties, this may go some way to explaining the success of DMD with small dictionaries.

It is worth noting however that the left Koopman modes $b_{j,K}$ converge in the space of hyper-distributions $\Hd^2_{-t}$ (dual to $\Hd^2_{t}$) to the generalised distribution $b_j$, which does not make sense as a function in the proper sense of the word, but rather as something you can integrate sufficiently smooth functions against. In particular, the $b_{j,K}$ do not converge pointwise or in $L^2$. This situation is illustrated in Figure~\ref{fig:eigmodes}: for low orders $K$ the right Koopman modes $b_{j,K}$ do in fact appear like ``nice'' enough functions that delimit almost-invariant sets, but as the resolution of the approximation (i.e. $K$) increases they become increasingly irregular: the consistency of the envelope of the functions in Figure~\ref{fig:eigmodes} nevertheless suggests that they under some smoothing they may consistently return functions delimiting almost-invariant sets.

\rev{It is also worth noting that Koopman operators of chaotic systems are typically highly non-normal in spaces of differentiable functions, and are not guaranteed to be diagonalisable. In our setting nevertheless, because the Koopman operator is are compact, all eigenvalues except $0$ admit a Jordan decomposition). Some explicit, uniform bounds for the error in eigenvalues are given in \cite{Bandtlow15}. 
}

The overarching strategy to prove Theorem~\ref{t:SimpleSpectral} parallels that sketched in \cite{Slipantschuk20}, i.e. we study by duality a transfer operator acting on Hilbert spaces $\{\Hd^2_t\}_{t \in (0,\zeta]}$ of analytic functions, and consider the effect of the least squares projection onto the finite observable space \eqref{eq:ObsDict} in these Hilbert spaces: Theorem~\ref{t:GalerkinGood} allows us to bound the error of $\P_K$ from our Hardy spaces $\Hd^2_t \to \Hd^2_s$ as $\mathcal{O}(e^{-(t-s)})$ for small enough $t,s$ for analytic measure densities $\md$, which is enough to obtain this result.
\\

\begin{figure}[]
	\centering
	\includegraphics[width=0.7142\linewidth]{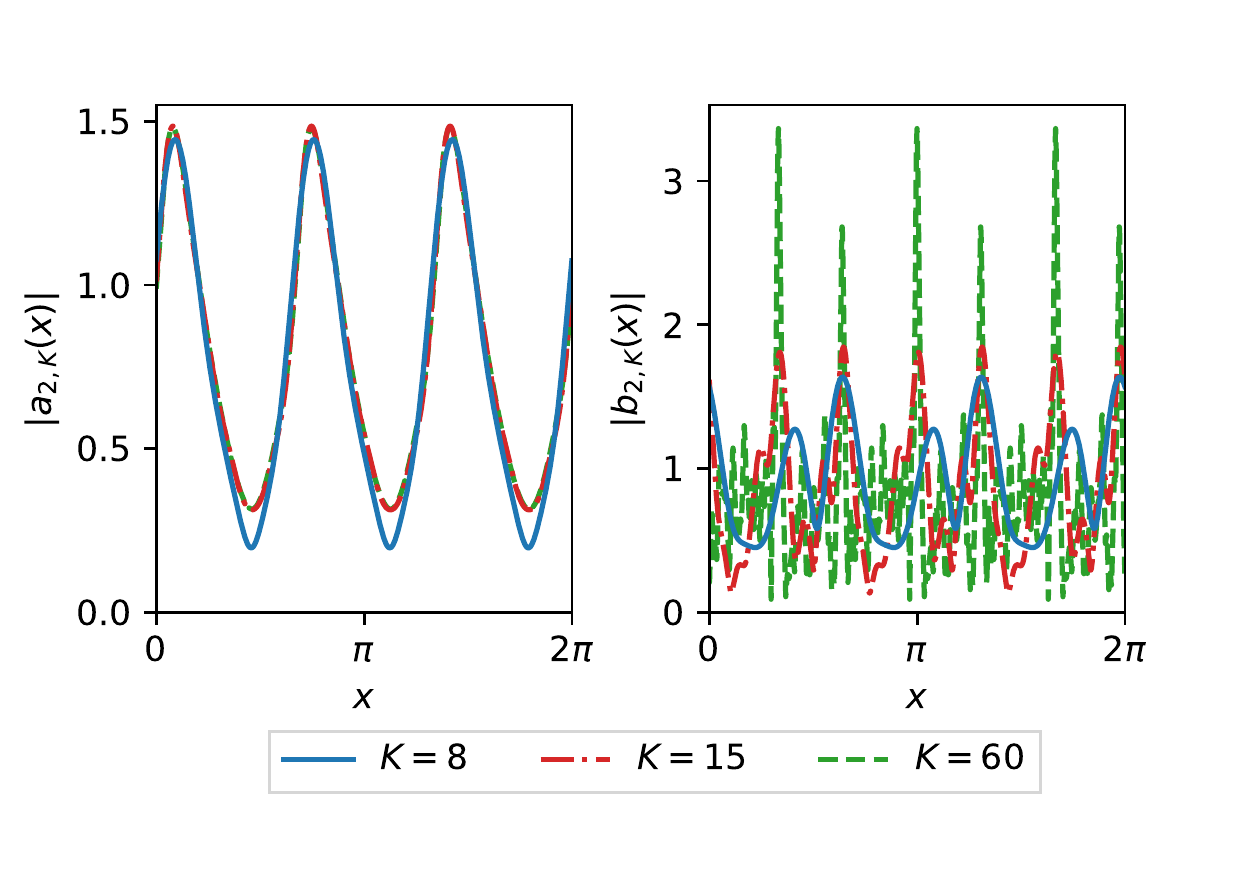}
	\caption{For the circle map given in Figure~\ref{fig:circlemap}, the absolute value of the left (left plot) and right (right plot) eigenmodes of the Koopman operator approximation $\K_K$ corresponding to the second-largest eigenvalue. The different modes of convergence (respectively in a space of analytic functions, and weakly, as a hyperdistribution) are shown.}
	\label{fig:eigmodes}
\end{figure}

We have been careful to make our results as tight as possible, but there are places where they could be meaningfully improved. In particular, in proving the least-squares approximation result, we bounded the operator norms of certain operators in Proposition~\ref{p:TildeOperators} via the trace norm, which would scale poorly into higher dimensions, and even in one dimension requires $\mu$ to be at least half an order more differentiable than the functions it is approximating. 

This would make it difficult to apply our methods to study EDMD with respect to less regular densities---which naturally occur as physical meaures of chaotic systems and therefore as the sampling measure for time series. A salutary example is for the logistic map $f(x) =ax(1-x)$, which is chaotic for a positive measure set of parameters $a \in [3.5,4]$, and which models the non-hyperbolic dynamics which are standard in fluid flow. Chaotic logistic maps typically have a physical measure whose density contains inverse square root singularities, and therefore lacks the level of differentiability required for our results \cite{Ruelle09}. Further work, perhaps building on OPUC theory in \cite{Geronimo06}, might show a path through, and establish how well DMD works with the kinds of irregular sampling measures typical of most chaotic dynamics.

In this work we have of course put aside an important, indeed limiting, component: the effect of the data discretisation. We know that this converges weakly as the empirical measure of the data $\{x_n\}$ converges weakly to the limiting measure $\dd\mu(x)$ \cite{Klus16}: there are many results showing this, including for time series data $x_{n+1} = f(x_n)$, where we know the rate of convergence is the central limit theorem rate $\mathcal{O}(N^{-1/2})$ \cite{Chernov95}. This is borne out in simulations: see Figure~\ref{fig:eigmodes}. Nevertheless, this rate of convergence may also depend on the size of the dictionary (i.e. on $K$), a key direction for future work.

\subsection{Paper plan}

This paper is structured as follows. In Section~\ref{s:Results} we make a precise statement of our assumptions on the system and the corresponding theorems. In Section~\ref{s:Cholesky} we relate the $L^2(\mu)$-orthogonal trigonometric polynomials to a Cholesky decomposition of multiplication by $\mu$ in a Fourier basis, and in Section~\ref{s:Ortho} quantitatively characterise this decomposition to prove Theorem~\ref{t:GalerkinGood}. Finally in Section~\ref{s:Transfer} we combine this result with transfer operator results to obtain our theorems on EDMD.

\begin{figure}
	\centering
	\includegraphics[width=0.5714\linewidth]{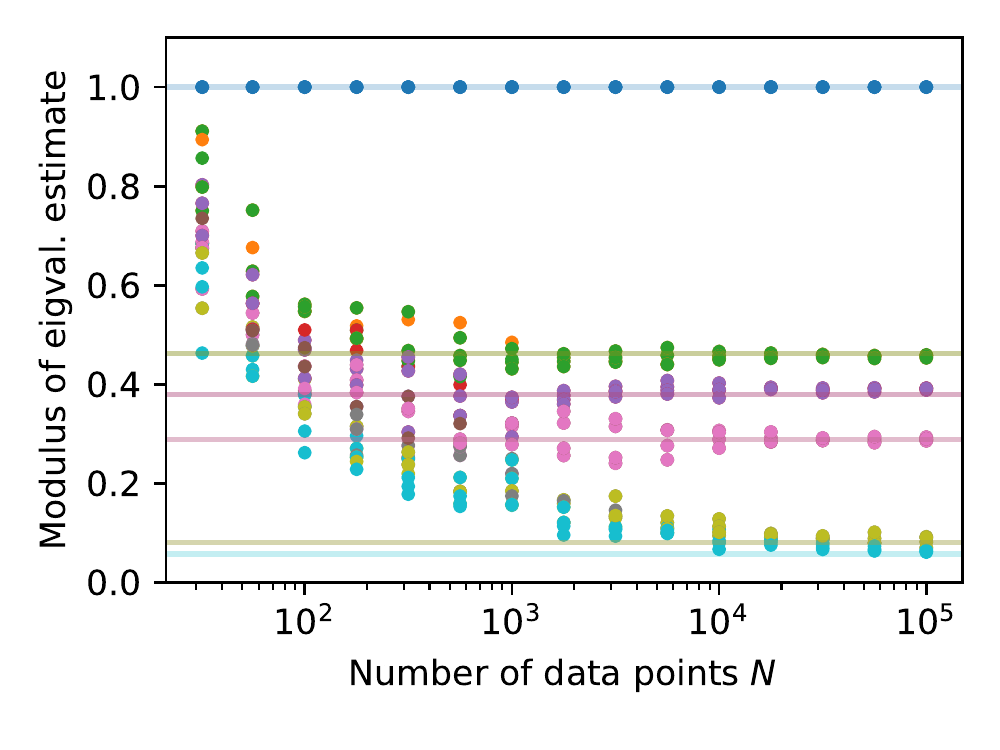}
	\caption{Convergence of the EDMD operator spectrum for fixed approximation size $K=8$ as the number of data points goes to infinity. In pale lines, the corresponding spectrum.}
	\label{fig:dataeig}
\end{figure}

\section{Detailed setup and theorems}\label{s:Results}
In this section we will state the precise condition on the map $f$ and sampling density $\mu$ that we require, and state our Koopman operator approximation theorems in generality.

We will base our function spaces on extensions of the standard $L^p$ spaces on the torus $\T = \R/2\pi\Z$. Recall the $L^p$ norms are defined as
\[ \| \varphi \|_{L^p(\T)} = \begin{cases} \left( \int_0^{2\pi} |\varphi(x)|^p \,\dd x \right)^{1/p},& p \in [1,\infty) \\ \esssup_{x\in \T} | \varphi(x) |, & p = \infty. \end{cases} \]

However, we will need to construct stronger function spaces to obtain effective results. In particular, we will look at on certain open complex strips in phase space around $\T$, which we will parametrise by their half-thickness $\zeta$:
	\[ \T_\zeta := \{z \in \C/2\pi\Z: |\Im z| < \zeta\}\]
On these sets we will define so-called Hardy spaces $\Hd^p_\zeta$ for $\zeta>0$: the space $\Hd^p_\zeta$ is the set of holomorphic functions $\varphi: \T_\zeta \to \C$ such that 
\[ \| \varphi \|_{\Hd^p_\zeta} := \begin{cases}
\sup_{u \in [0,\zeta)} \left(\frac{1}{4\pi} \int_\T |\varphi(\theta+i u)|^p + |\varphi(\theta - iu)|^p\,\dd\theta\right)^{1/p},& p \in [1,\infty)\\
\sup_{z \in \T_\zeta} |\varphi\rev{(z)}|,& p = \infty \end{cases} \]
is finite (in which the supremum is attained in the limit as $u \to \zeta$, and $\varphi$ is continuous onto $\partial \T_\zeta$ almost everywhere). The $L^p(\T)$ norms emerge as a limiting case as $\zeta \to 0$. (For concision, we will define $\Hd^p_0:=L^p(\T)$).

The spaces $\Hd^2_\zeta$ are Hilbert spaces, and we can characterise these norms in Fourier space. Let $\F: L^2(\T) \to \ell_2(\Z)$ be the Fourier series operator:
\[ (\F\varphi)(k) = \frac{1}{2\pi} \int_\T e^{-ikx} \varphi(x)\,\dd x \]
and for $\sigma$ a Beurling weight let us define the Hilbert spaces $W^\sigma(\T) = \F^{-1} [\sigma \ell^2(\Z)]$ with
$$ \|\varphi\|_{W^\sigma} = \| \sigma \F \varphi \|_{\ell^2(\Z)}. $$

It is a classic result that the standard $L^2(\T)$ is such a space with $\sigma(n) \equiv 1$; Sobolev spaces $H^r$ are isometric to those with $\sigma(n) = (1+n^2)^{r/2}$. Similarly, if we define the Beurling weight
\[\sigma_\zeta(k) := \sqrt{\cosh 2k\zeta},\] 
then our Hardy space \rev{has the construction\footnote{\rev{This works as follows. The complex exponential functions are orthogonal in both $\Hd^2_\zeta(\T)$ and $W^{\sigma_\zeta}$ with $\langle e_j, e_k \rangle = \delta_{jk} \cosh 2k\zeta = \sigma_\zeta(k)^2$, so we have that $W^{\sigma_\zeta}$ (which is the span of the complex exponentials) is a closed subspace of $\Hd^2_\zeta(\T)$. But this means that any element of $\Hd^2_\zeta$ orthogonal to $W^{\sigma_\zeta}$ must have all zero Fourier coefficients, so must be identically zero: hence they are the same.}}} $\Hd^2_\zeta(\T) = W^{\sigma_\zeta}$.

We can also define spaces that act as the dual of these smooth function spaces. For a Beurling weight $\sigma$ let $W^{\sigma^{-1}}$ be the completion of $L^2(\T)$ with respect to the following norm:
\[ \| \varphi \|_{W^{\sigma^{-1}}} = \sup_{\|\psi\|_{W^{\sigma}} = 1} \frac{1}{2\pi} \int_\T \bar\varphi\, \psi\,\dd\theta. \] 
This Hilbert space $W^{\sigma^{-1}}$ is therefore isometric to $\sigma^{-1} \ell^2(\mathbb Z)$ under the Fourier series operator $\F$. We can of course therefore define $\Hd^2_{-\zeta} = W^{\sigma_\zeta^{-1}}$. \rev{It is isometric to the dual of $\Hd^2_\zeta$ by the map taking distribution $\varphi$ to $\psi \mapsto \tfrac{1}{2\pi} \int_\T \varphi\, \psi\,\dd x$. }

We will use the notation $\mathcal{A}^\dag$ to stand for the adjoint operator of $\mathcal{A}$ in $L^2(\T,\mathrm{\Leb}/2\pi)$, and extend this to functions: $\phi^\dag\psi := \int_\T\bar{\phi}\, \psi\,\dd x$. 
\\

Our basic assumptions for our results on EDMD will be that the map $f$ is uniformly expanding and real-analytic and the sampling measure $\mu$ has positive density $\md(x)$ with respect to Lebesgue. Note that if $\mu$ is the absolutely continuous invariant measure of $f$, then the first part implies the second.

To better understand rates of convergence, we will make the following quantitative assumptions on $f$ for some $\zeta > 0$:
	\begin{itemize}
		\item The {\rev periodic lift} $\hat f: \R \to \R$ of $f$ has an inverse $v: \R\to\R$ that extends analytically to $\R_\zeta := \R + i[-\zeta,\zeta]$. {\rev (Supposing $f$ is $w$-to-one, $\hat f$ will map $[0,2\pi]$ onto an interval of length $2\pi w$, and its inverse $v$ will map $[0,2\pi w]$ onto an interval of length $2\pi$.)}
		\item For $z \in \R_\zeta$, $v$ is a uniform contraction with $|v'(z)| \leq \kappa < 1$.
		\item $\| \log (v') \|_{C^\alpha(\T_\zeta)} \leq D$ for some constants $\alpha > 0$ and $D \leq \tfrac{\pi}{3} \zeta^{-\alpha}$.
	\end{itemize}
The last bound can always be achieved by making $\zeta$ smaller.

When studying the convergence of EDMD, we will also make the following quantitative assumptions on the measure density $\md$:
	\begin{itemize}
		\item $\md$ extends analytically to $\T_\zeta$.
		\item For $z \in \T_\zeta$, $M^{-1} \leq |\md(z)| \leq M$ for some constant $M$.
		\item $\|\sigma \F[(\log \md)']\|_{\ell^q} \leq A$ for some constants $A$ and $q <\infty$. (Note this occurs if $(\log \md)' \in L^p(\partial \T_\zeta)$ where $1/p + 1/q = 1$.)
	\end{itemize}

The key to our results on spectral convergence is the following strong operator convergence result:
\begin{theorem}\label{t:Deterministic}
	For all $t \in (0,\zeta]$ the Koopman operator $\K$ extends to an operator $\Hd^2_{-t} \circlearrowleft$. On this space it is compact and there exists a constant $C$ depending only on $M,A,D,\zeta$ such that
	\[ \|\P_K\K - \K \|_{\Hd^2_{-t}}, \|\K\P_K - \K \|_{\Hd^2_{-\kappa t}}  \leq 
	C e^{-(1-\kappa) t K}. \]
\end{theorem}

Because the spectrum and right (resp. left) eigenmodes of $\P_K \K$ (resp. $\K \P_K$) are the same as that of $\K_K$, this allows us to estimate all the eigenmodes and Ruelle-Pollicott resonances of the system in \eqref{eq:RuellePollicott}.

\rev{Define the distance between two closed vector spaces $E,F$ in a norm $\| \cdot \|$ as
\[ d(E,F) = \max\left\{ \sup_{e\in E, \|e\|=1} d(e,F), \inf_{f \in F,\|f\|=1} d(f,E) \right\} \]
where the distance from a vector to a vector subspace is defined as $d(e,F) = \inf_{f \in F} \| e- f\|$. Then we have the following theorem:}
\begin{corollary}\label{c:Spectrum}
	For all $\K$'s eigenvalues $\lambda$ of multiplicity $m$, the error in the $\K_K$ estimate as $K \to \infty$ is $\mathcal{O}(e^{-(1-\kappa)\zeta K/m})$ for the eigenvalues, and $\mathcal{O}(e^{-(1-\kappa) t K})$ in $\Hd^2_{-t}$ (resp. $\Hd^2_{\kappa t}$) for the right (resp. left) generalised eigenspaces. 
\end{corollary}

The proofs of these results are given in Section~\ref{s:Transfer}.


\section{Cholesky factorisation}\label{s:Cholesky}

To understand the $\mu$-orthogonal projection onto our dictionary of complex exponential basis functions $\{e_k\}_{|k|\leq K}$, it will be natural to understand the transformation between the complete, infinite basis $\{e_k\}_{k \in \Z}$ and an orthogonal polynomial basis of $L^2(\mu)$ (in which the $L^2(\mu)$ projection is therefore diagonal). In the correct ordering it is known that the basis change matrices are triangular and asymptotically banded \cite{Simon05}: we will prove some quantitative results on the rate of convergence to bandedness as the order of the polynomial grows.

Recall that the complex exponential family on $\C/2\pi\mathbb{Z} \supset \T$ 
\[ e_k(z) := e^{ikz},\, k \in \Z, \]
is a basis both of $L^2(\T)$ and our Beurling-weighted spaces $W^\sigma$. For the purposes of this section, let us order the basis' index set $\Z$ as
\[ 0, -1, 1, -2, 2, -3, 3, \ldots\]
and group the basis elements into ``blocks'' $\{e_0\}, \{e_{-1},e_1\}, \{e_{-2},e_2\},\ldots$.


For any \rev{bounded} probability density $\md: \T \to \R_+$, define the Fourier space multiplication operator $\M: L^2(\T) \to L^2(\T)$:
\begin{equation} \M \varphi = \md \varphi. \label{eq:MDef} \end{equation}


The operator $\M$ is positive-definite and Hermitian in $L^2(\T)$. (In fact, in Fourier space it is near-block-Toeplitz with respect to our basis order.) It turns out we can make a Cholesky decomposition of $\M$ that is very nicely, uniformly bounded in our Hardy or Beurling-weighted Hilbert spaces, under our analyticity assumptions on $\md$:
\begin{proposition}\label{p:CholeskyExists}
 Supposing only that \rev{$\md$ is bounded above and away from zero}, there exists a unique lower triangular operator $\V$ acting on $L^2(\T)$ such that
			\[ \M = \V \V^\dag. \]
		Furthermore, $\V$ is invertible. If we let $\V^{-1} =: \U^\dag$, the operator $\U$ is upper-triangular with
			\[ \M^{-1} = \U^\dag \U. \]
\end{proposition}

			\begin{proof}[Proof of Proposition~\ref{p:CholeskyExists}]
				Let us consider $\M$ as an infinite matrix acting on the basis \[1, \sin z, \cos z, \sin 2z, \cos 2z, \ldots\] The matrix of $\M$ is real and positive-definite, bounded on $L^2(\T)$. By \cite[Lemma~3.1]{Chui82}, there exists a unique lower-triangular operator $\V_\R$ such that $\M= \V_\R \V_\R^\dag$. Now, for any function $\varphi \in L^2(\T)$,
				\begin{equation} \varphi^\dag \M \varphi = \varphi^\dag \V_\R \V_\R^\dag \varphi = \| \V_\R^\dag \varphi \|_{L^2}^2 \label{eq:CholeskyNormTrick1}\end{equation}
				so since $\M \varphi = \md \varphi$, 
				\begin{equation} M^{-1} \|\varphi\|_{L^2}^2 \leq \| \V_\R^\dag \varphi \|_{L^2}^2 \leq M \|\varphi\|_{L^2}. \label{eq:CholeskyNormTrick2}\end{equation}
				Since, $\V^\dag_\R$ is bounded and invertible in $L^2(\T)$, so too is its adjoint $\V_\R$. Now, $\V^\dag_\R$ is block-lower triangular in the complex exponential basis $e_k$, the blocks consisting of $\{e_k, e_{-k}\}$. We can perform a $LQ$ decomposition on each block to obtain a block-diagonal unitary matrix $\mathcal{Q}$ such that $\V = \V_\R \mathcal{Q}$  is lower-triangular in the complex exponential basis, and \rev{$\U = (\V^{-1})^\perp$} is upper-triangular in this basis with $\V \V^\dag = \M$ and $\U^\dag \U= \M^{-1}$.
			\end{proof}

Now, define the functions 
\[ p_k = \U e_k. \]
From the proposition above, it turns out that these are a complete family of (complex) trigonometric polynomials orthonormal with respect to $\md$:
\begin{proposition}\label{p:Polynomials}
	The $p_k$ are each trigonometric polynomials of order exactly $|k|$, and form a complete orthonormal basis in $L^2(\mu)$.
\end{proposition}
\begin{proof}
	Integrating two polynomials against each other,
	\begin{align*}
		\int_\T \overline{p_j} p_k \,\dd \mu &= (\U e_j)^\dag \md \U e_k\\
		&= e_j^\dag \U^\dag ((\U^\dag)^{-1} \U^{-1}) \U e_k\\
		&= e_j^\dag e_k = \delta_{jk},
	\end{align*}
	as required for orthonormality. Because $\U$ is upper-triangular in the complex exponential basis with non-zero diagonals, $\U e_k$ is a combination of complex exponentials of order no greater than $|k|$, so $p_k$ is a trigonometric polynomial of the right order. 
	
	On the other hand, the complex exponentials $\{e_k\}_{k\in\Z}$ form a complete basis of $L^2(\T,\mathrm{\Leb})$ and hence of $L^2(\T,\mu)$, as the spaces are equivalent. But the $p_k$ are preimages of $e_k$ under $\U^{-1}: L^2(\T, \mu) \circlearrowleft$, so $\{p_k\}_{k\in\Z}$ must therefore, like $\{e_k\}_{k\in\Z}$, be a complete basis of $L^2$.
\end{proof}

The $p_k$ as described here are complex, but they can be transformed via a unitary map into a real orthonormal family (see the proof of Proposition~\ref{p:CholeskyExists}): we have only chosen a complex exponential basis for the simplicity of the following presentation.

%

The corollary of Proposition~\ref{p:CholeskyExists} is that $L^2(\mu)$-orthogonal projection is conjugate to the usual $L^2(\Leb)$ projection (i.e. the Dirichlet kernel, whose properties are very well known):
\begin{proposition}\label{p:OrthoProjectionDecomposition}
	Let $\P_K: L^2(\mu) \circlearrowleft$ be the orthogonal projection onto trigonometric polynomials of order $< K$. Then
	\[ \P_K = \U^{-1} \D_K \U = \V^\dag \D_K \U \] 
	where $\D_K = \F^{-1} \mathbb{1}_{(-K+1,\ldots,K-1)} \F$ is convolution by the Dirichlet kernel.
\end{proposition}
\begin{proof}
Since $\{e_k\}_{k\in\Z}$ is a complete basis of $L^2(\T)$, and the $p_k$ are obtained from the $e_k$ by the transformation $\U$, which from Proposition~\ref{p:CholeskyExists} we know is invertible on $L^2(\T)$, our polynomials $\{p_k\}_{k \in \Z}$ form a complete basis of $L^2(\Leb) = L^2(\mu)$, setting $\sigma\equiv 1$. (Indeed, applying Theorem~\ref{t:CholeskyBound} below instead of Proposition~\ref{p:CholeskyExists}, we can say that this is true in the Beurling weighted spaces $W^\sigma$.)

The action of $\P_K$ on our basis is $\P_K p_k = \mathbb{1}_{|k| < K} p_k$. Using from Proposition~\ref{p:Polynomials} that $p_k = \U e_k$, we can conjugate by $\U$ to obtain
\[ \U^{-1} \P_K \U e_k = \mathbb{1}_{|k| \leq K} e_k. \]
Since the action of $\U^{-1} \P_K \U$ is precisely the action of $\D_K$, and $\U^{-1} = \V^\dag$, we obtain that  $\P_K = \U^{-1} \D_K \U = \V^\dag \D_K \U$ as required.
\end{proof}

Thus motivated to study Cholesky decompositions of multiplication operators, we embark on proving that such things exist and their constituents are bounded in $W^\sigma$.

\section{Proof of orthogonal polynomial results}\label{s:Ortho}

We would like to understand our orthogonal projection $\P_K$ in the Beurling weighted spaces $W^\sigma$. Proposition~\ref{p:OrthoProjectionDecomposition} suggests that we should therefore study the norm of triangular operators $\U, \V$ in these weighted spaces. The following result, which we will spend this section proving, shows that these operators are uniformly bounded:
		\begin{theorem}\label{t:CholeskyBound}
	Under the assumptions on $\mu$ in Theorem~\ref{t:GalerkinGood}, there exists a constant $C_{\triangle}$ increasing in $q, M, A$ such that
	\[ \| \U \|_{W^\sigma}, \| \V \|_{W^\sigma}, \| \U^\dag \|_{W^\sigma}, \| \V^\dag \|_{W^\sigma} \leq C_{\triangle}. \]
\end{theorem}

Our goal is to show that these operators $\U, \V$ obtained by Cholesky decomposition of a multiplication operator $\M$ approximate multiplication operators themselves, at least when acting on functions of high frequency.

To this end, let $\theta^+, \theta^-$ be real-analytic functions $\T \to \C$ such that $\theta^-(\bar z) = \overline{\theta^+(z)}$, $\theta^+ \theta^- = 1/\md$ and $\theta^+$ is holomorphic in the upper half-plane. We can specify them explicitly:
\begin{equation} (\F \log \theta^+)(k) = \left(\half \delta_{0k} + \mathbb{1}(k > 0)\right)\F(-\log \md)(k).\label{eq:ThetaDef}\end{equation}
and similarly for $\theta^-$ with $\mathbb{1}(k<0)$ replacing $\mathbb{1}(k>0)$. In the language of orthogonal polynomials on the unit circle (where one studies the variable $z = e^{ix}$), the function $\theta^+(\log z)$ is known as the {\it Szeg\H{o} function} of $\dd\mu(\log z)$ \cite{Simon05}. We will also notate their reciprocals $\eta^\pm = 1/\theta^\pm$. The following result states some of the basic properties of $\theta^\pm, \eta^\pm$ (c.f. \cite[Theorem~2.4.1]{Simon05}).

\begin{proposition}\label{p:EtaBounds}
	Under our assumptions on $\md$:
	\begin{enumerate}[a.]
	\item For all $x \in \T$ we have 
	\[ |\theta^+(x)| = |\theta^-(x)| = \sqrt{\md(x)}. \]
	 \item Considering functions as multiplication operators,
	\[ \| \theta^\pm\|_{W^\sigma\to W^\sigma}, \| \eta^\pm \|_{W^\sigma \to W^\sigma}\leq M^{1/2} e^{qA}. \]
	\item\label{res:ThetaSymbol} If $\hat \theta^+_l$ are the Fourier coefficients of $\theta^+$, then
	\[ \sum_{l=0}^\infty l \sigma(l)^2 |\hat \theta^+_l|^2 \leq q A^2 M e^{2qA}.  \]
	\end{enumerate}
\end{proposition}

We will briefly find it useful to notate some $\l^p$-type norms associated with the Beurling weights. For $q \in [1,\infty]$ let $\|\varphi\|_{\sigma;q} = \|\sigma \F\varphi\|_{\ell^q}$.

\begin{lemma}\label{l:BeurlingMultiplication}
	Suppose $\sigma$ is a Beurling weight. Then for any $q \in [1,\infty]$,
	$$ \| \varphi \psi \|_{\sigma;q} \leq  \| \varphi \|_{\sigma;q}  \| \psi \|_{\sigma;1} $$
	whenever $\|\varphi\|_{\sigma;r}, \|\psi\|_{\sigma;1} < \infty$.
\end{lemma}
\begin{proof}
	All we need to prove is that
	$ \| \sigma (\hat\varphi \ast \hat\psi) \|_{\ell^q} \leq \| \sigma \hat\varphi \|_{\ell^q}  \| \sigma \hat\psi \|_{\ell^q}.$
	
	By the definition of the Beurling weight, $\sigma(k) \leq \sigma(|j|+|k-j|) \leq \sigma(j) \sigma(k-j)$, so
	\begin{align*} |\sigma(k) (\hat\varphi \ast \hat\psi)(k)| &= \sum_{j\in\Z} \sigma(k) |\hat\varphi(j)||\hat\psi(k-j)|\\
		 &\leq \sum_{j\in\Z} |\sigma(j)\hat\varphi(j)| |\sigma(k-j)\hat\psi(k-j)| \\
		 &\leq (|\sigma \hat\varphi| \ast |\sigma \hat\psi|)(k)
		 \end{align*}
	 As a result,
	\[ \| \sigma (\hat\varphi \ast \hat\psi) \|_{\ell^q} \leq \|  |\sigma\hat\varphi| \ast |\sigma\hat\psi| \|_{\ell^q} \leq \| \sigma \hat\varphi \|_{\ell^q}  \| \sigma \hat\psi \|_{\ell^1}\]
	as required.
\end{proof}

\begin{proof}[Proof of Proposition~\ref{p:EtaBounds}]
	We only need to prove the results for $\theta^+$, as $\log \theta^+(\bar z) = \overline{\log \theta^-(z)}$.
	
	The first part directly uses this fact: $|\theta^+(x)|^2 = \theta^-(x) \theta^+(x) = \md(x)^{-1}$.
	
	To prove the second part, we separate $k=0$ and apply H\"older's inequality to get 
	\[ \| \log \md \|_{\sigma;1} \leq \sigma(0) \left| \int \log \md\,\dd x \right| + \| \sigma(k) k \mathcal{F}[\log \md](k) \|_{\ell^q} \| \mathbb{1}_{k\neq 0} k^{-1}\|_{\ell^p} \]
	where $1/p + 1/q = 1$. Noting that by submultiplicativity $\sigma(0) \leq 1$, this gives us
	\[ \|\log \md \|_{\sigma;1} \leq \log M + \| (\log \md)'\|_{\sigma; q} (2\zeta(p))^{1/p} \leq \log M + A (2q)^{1/p} \leq \log M + 2qA. \]
	Noting that $\log \md$ is real on $\T$, \eqref{eq:ThetaDef} gives us that 
	\[ \| \log\theta^+ \|_{\sigma; 1} = \tfrac{1}{2}\| \log \md \|_{\sigma; 1} = \log M^{1/2} + qA. \]
	Now, $\theta^+$ and $\eta^+=\tfrac{1}{\theta^+}$ are respectively the $t=1,-1$ solutions of
	\[ \frac{\partial}{\partial t} \mathcal{E}_t(x) = \log\theta^+\, \mathcal{E}_t(x),\,  \mathcal{E}_0(x)  = 1\]
	so Gronwall's Lemma combined with Lemma~\ref{l:BeurlingMultiplication} gives
	\[ \|\theta^+ \|_{\sigma; 1},  \| \eta^+ \|_{\sigma; 1} \leq e^{\| \log\theta^+ \|_{\sigma;1}} \| 1 \|_{\sigma;1} \leq e^{qA} M^{1/2}. \]
	The bounds on $\theta^+,\eta^+$ considered as multipliers on $W^\sigma$ (which has the $\|\cdot\|_{\sigma;2}$ norm) follows from Lemma~\ref{l:BeurlingMultiplication}.
	
	For the third part we proceed in the same vein as the second part. From \eqref{eq:ThetaDef} we have $\F(\log \theta_+)' = \mathbb{1}(k > 0)\F(-\log \md)(k)$, so $\|\sigma \F[(\log \theta_+)'] \|_{\ell^q} \leq A$. Then,
	 \[\| (\theta^+)'\|_{\sigma;q} =\| (\log \theta^+)'\|_{\sigma;q} \| \theta^+\|_{\sigma;1} \leq A M^{1/2} e^{qA}.\]
	 Applying H\"older's inequality gives that
	 \begin{align*} \|k^{1/2} \sigma \mathcal{F}[\theta^+]\|_{\ell_2}^2 &= \| ( \mathbb{1}(k\neq 0) k^{-1/2}) \sigma \mathcal{F}[(\theta^+)']\|_{\ell^2}^2\\
	 	& \leq \| \mathbb{1}(k\neq 0) k^{-1} \|_{\ell^{q/(q-2)}} \| (\theta^+)' \|_{\sigma; q}^2 \\
	 	& \leq q A^2 M e^{2qA},
 	\end{align*}
	which is what needed to be proven.
\end{proof}

With these properties in hand we can now characterise the asymptotically banded structure of the Cholesky decomposition.

Define the projections
\begin{align*} \P^+ &= \F^{-1} \mathbb{1}_{\Z^+} \F\\
\P^\circ &= \F^{-1} \mathbb{1}_{\{0\}} \F = e_0 e_0^\dag\\
\P^- &=  \F^{-1} \mathbb{1}_{\Z^-} \F.
\end{align*}

The limiting Cholesky factors, which describe the action of $\U,\V$ on $e_k$ with $|k|$ large, are
\[ \bU = \bc \P^\circ + \P^+ \theta^- + \P^- \theta^+. \]
\[ \bV = \bc^{-1} \P^\circ + \eta^+ \P^+ + \eta^- \P^-  \]
where $\bc$ is an arbitrary constant: we set $\bc =\sqrt{\tfrac{1}{2\pi}\int \md^{-1}\,\dd x}$.

These obey the same relation as their equivalents $\V, \U$ do:

\begin{lemma}\label{l:BVBURelation} $\bV^{-1} = \bU^\dag$. \end{lemma}
\begin{proof}
	If $\psi^+$ contains only non-negative Fourier modes, then $\psi^+ \P^+ = \P^+ \psi^+\P^+$, with the corresponding result for functions with non-positive modes. 
	
	We have that $\bU^\dag = \bc \P^\circ + \theta^+ \P^+  + \theta^- \P^-$, and the result follows by expanding out $\bU^\dag \bV$ and $\bV \bU^\dag$.
\end{proof}

These operators are uniformly bounded:
\begin{lemma}\label{l:GraveOperators}
	For all $t \in [0,\zeta]$,
	\[ \|\bU\|_{W^\sigma}, \|\bU^\dag\|_{W^\sigma}, \|\bV\|_{W^\sigma}, \|\bV^\dag\|_{W^\sigma} \leq (1 + 2e^{qA}) M^{1/2}. \]
\end{lemma}
\begin{proof}
	Since the projections $\P^+, \P^-, \P^\circ$ have norm $1$ in $W^\sigma$,
	\[ \|\bU\|_{W^\sigma} \leq \bc + \|\theta^+\|_{W^\sigma\to W^\sigma} + \|\theta^-\|_{W^\sigma\to W^\sigma} \leq (1 + 2e^{qA}) M^{1/2}. \]
	The same argument goes for the other operators.
\end{proof}

%

\begin{lemma}\label{l:SwappingPAndFunction}
	If $\P^+ w = 0$, then $\P^+ w - w \P^+ = (\P^- + \P^\circ) w \P^+$.
	The same statement holds swapping $\P^+$ and $\P^-$.
\end{lemma}
\begin{proof}
	For $k \leq 0$ the operator $\P^+ e^{-ik\cdot} \P^+$ reduces to $\P^+ e^{ik\cdot}$. Since we have that $w(z) = \sum_{k=0}^\infty \hat w_{-k} e^{-ikz}$ so $\P^+ w \P^+ = \P^+ w$. Then, 
	\[ \P^+ w - w \P^+ = (I - \P^+) w \P^+ = \P^\circ w \P^+ + \P^- w \P^+. \]
\end{proof}

%

We will want to show that in some sense $\bV$ is a good approximation of $\V$, the Cholesky factor of $\M$, and similarly $\bU$ is a good approximation of $\U$. We can understand this by attempting to ``strip'' $\M$ down to the identity, by considering $\bV^\dag \M^{-1} \bV$ and $\bU^\dag \M \bU$. (It is useful to study both at the same time: in Proposition~\ref{p:TildeOperators} we will use one-sided bounds on each one to complete the bounds on the other.) 

It will turn out in Proposition~\ref{p:TildeOperators} that all we need to know are the diagonal entries of these operators. In the following two propositions, we will show that these entries converge to those of the identity (c.f. \cite[Proposition~2.4.7, Theorem~7.2.1]{Simon05}).

\begin{lemma}\label{l:SDashAreZero}
	For all $k \in \Z$,
	\[ s'_k = e^\dag_k \bV^\dag \M^{-1} \bV e_k - 1 = 0. \]
\end{lemma}
\begin{proof}
	We have
	\[ 1 + s'_k = (\bV e_k)^\dag \M^{-1} (\bV e_k). \]
	When $k=0$, 
	\[ \bV e_0 = \bc^{-1} e_0 = (e_0^\dag\md^{-1})^{-1/2} e_0 \]
	so 
	\[ 1 + s'_0 = (e_0^\dag \md^{-1})^{-1} e_0^\dag \md^{-1} e_0 = 1. \]
	On the other hand, when $k>0$ (the $k<0$ case follows by the same argument), $ \bV e_k = \eta^+ \P^+ e_k = \eta^+ e_k$ 
	so
	\[ 1 + s'_k = e_k^\dag \eta_- \eta_- \eta_+ \eta^+ e_k = e_k^\dag e_k = 1, \]
	as required.
\end{proof}

\begin{lemma}\label{l:SDecay}
	Let 
\[ s_k = e^\dag_k \bU^\dag \M \bU e_k - 1. \]
Then $s_k > -1$ and there exists a constant $C$ increasing in $M,A,q$ such that
	\begin{equation} \sum_{k\in \Z} \sigma(k)^2 |s_k| \leq C. \label{eq:DiagonalValueSum}\end{equation}
\end{lemma}
\begin{proof}
	In this case we have 
	\[ 1 + s_k = (\bU e_k)^\dag \M (\bU e_k), \]
	but the situation is more complicated than in Lemma~\ref{l:SDashAreZero} because, when compared with the previous proposition, the order of projection and multiplication in the definition of $\bU$ are now reversed compared with $\bV$. We can use Lemma~\ref{l:SwappingPAndFunction} to swap them back, but this means we pay the price of a small error, which we now bound.
	
	For $k=0$ we have that 
	\[ \bU e_0 = \P^+ \theta_- e_0 + \P^- \theta_+ e_0 + \bc \P^\circ e_0 = \bc e_0  \]
	so 
	\[ s_0 = 1 - \bc^2 e_0^\dag \md e_0 = 1 - \frac{1}{2\pi}\int_T \md^{-1}\,\dd x\, \frac{1}{2\pi} \int_\T \md\,\dd x \]
	from which one may extract that $|s_0| \leq (M-1)^2$. 
	
	For $k>0$ we have that
	\[ \bU e_k = \P^+ \theta_- e_k + \P^- \theta_+ e_k + \bc \P^\circ e_k = \P^+ \theta_- e_k, \] 
	and since $\M = \eta^+ \eta^- = \overline{\eta^-} \eta^-$,
	\begin{align*} (\bU e_k)^\dag \M (\bU e_k) &= (\eta^- \P^+ \theta^- e_k)^\dag \eta^- \P^+ \theta^- e_k.
	\end{align*}
	Now, we have that 
	\[ \eta^- \P^+ \theta^- e_k - e_k = \eta^- (\P^+ \theta^- - \theta^- \P^+) e_k = - \eta^- (\P^\circ + \P^-) \theta^- \P^+ e_k \]
	with the last equality by Lemma~\ref{l:SwappingPAndFunction}. Now,
	\[ (\eta^- \P^+ \theta^- e_k - e_k)^\dag e_k = - (\theta^- e_k)^\dag (\P^\circ + \P^-) \eta^+ \P^+ e_k = 0, \]
	so
	\begin{align*}	 (\bU e_k)^\dag \M (\bU e_k) &= e_k^\dag e_k + (\eta^- \P^+ \theta^- e_k - e_k)^\dag (\eta^- \P^+ \theta^- e_k - e_k)\\
	&= 1 + \|  -\eta^- (\P^\circ + \P^-) \theta^- \P^+ e_k \|^2_{L^2(\T)}\\
	&\leq 1 + \|\eta^-\|_{L^\infty(\T)}^2 \|(\P^\circ + \P^-) \theta^- e_k \|^2_{L^2(\T)}.
	\end{align*}
	Now,
	\[ \|(\P^\circ + \P^-) \theta^- e_k \|^2_{L^2(\T)} = \sum_{j = 0}^\infty |\F(\theta^- e_k)(-j)^2| = \sum_{j=0}^\infty |\hat \theta^-_{-j-k}|^2 = \sum_{l=k}^\infty |\hat \theta^+_{l}|^2. \]
	Consequently, and using the first part of Proposition~\ref{p:EtaBounds} to bound $\eta^-=\overline{1/\theta^+}$, we have 
	\[ s_k = |(\bU e_k)^\dag \M (\bU e_k) - 1| \leq M \sum_{l=k}^\infty |\hat \theta^+_{l}|^2. \]
	Combining the cases $k=0$, $k<0$, $k>0$ we find that for all $k \in \Z\backslash \{0\}$,
	\[ s_k \leq M \sum_{l=|k|}^\infty |\hat \theta^+_{l}|^2, \]
	which we can use to bound \eqref{eq:DiagonalValueSum}. In particular,
	\begin{align*}
	 \sum_{k\in \Z} \sigma(k)^2 |s_k| &\leq (M-1)^2 + 2 M \sum_{k=1}^\infty \sigma(k)^2 \sum_{l=k}^\infty |\hat \theta^+_{l}|^2 \\
	 &= (M-1)^2 + 2M \sum_{l=1}^\infty \left(\sum_{k=1}^l \sigma(k)^2\right) | \hat \theta^+_{l}|^2\\
	 &\leq (M-1)^2 + 2M \sum_{l=1}^\infty l \sigma(l)^2 | \hat \theta^+_{l}|^2 \leq M^2 (1 + qA^2 e^{2qA}), \\
	\end{align*}
	 using the last part of Proposition~\ref{p:EtaBounds} in the last part.
\end{proof}

Now we can use these bounds to get strong bounds on $\tilde V, \tilde V^\perp, \tilde U, \tilde U^\perp$.

\begin{proposition}\label{p:TildeOperators}
	Let $\V \bV^{-1} = I + \tilde\V$ and $\U \bU^{-1} = I + \tilde \U$. Then there exists $C$ increasing in $M,A,q$ such that
	\[ \|\tilde\V\|, \|\tilde\V^\dag\|, \|\tilde\U\|, \|\tilde\U^\dag\| \leq C \]
	in the $L^2(\T) \to W^\sigma$ operator norm.
\end{proposition}
\begin{proof}[Proof of Proposition~\ref{p:TildeOperators}]
Firstly, this is well-posed since by Lemma~\ref{l:BVBURelation}, $\bV^{-1} = \bU^\dag$ and $\U^{-1} = \bV^\dag$.

Let's notate the diagonal entries of $\tilde\V$ as $\alpha_k = e_k^\dag \tilde\V e_k$. Because $\V$ and $\bV$ have positive diagonal entries, so does $I + \tilde\V$, and hence the $\alpha_k > -1$.

Then for all $k \in \Z$,
\[ e_k^\dag (I+\tilde\V)(I+\tilde\V)^\dag e_k = \| \tilde \V^\dag e_k \|^2_{L^2} + 1 + 2 \alpha_k. \]
But we also know that 
\[ e_k^\dag (I+\tilde\V)(I+\tilde\V)^\dag e_k = e_k^\dag \bU^\dag \M \bU e_k = 1 + s_k \]
so
\begin{equation} 2 \alpha_k \leq \| \tilde \V^\dag e_k \|^2_{L^2} + 2 \alpha_k = s_k. \label{eq:V_alpha_s} \end{equation}

On the other hand let 
\[ \bV^\dag \M^{-1} \bV = (I+\tilde\U)(I+\tilde\U)^\dag.\]
Because $(\bU^\dag \M \bU)^{-1} = \bV \M^{-1} \bV$, we have that $I + \tilde\U = (I + \tilde\V)^{-1}$ and so the diagonal elements of $\tilde\U$ are $(1+\alpha_k)^{-1} -1 = -\tfrac{\alpha_k}{1+\alpha_k}$. Hence
\[ 1 + s'_k = e_k^\dag \bV^\dag \M^{-1} \bV e_k = e_k^\dag (I+\tilde\U)(I+\tilde\U)^\dag e_k = \| \tilde \U^\dag e_k \|^2_{L^2} + 1 - 2\tfrac{\alpha_k}{1+\alpha_k}, \]
so
\begin{equation}
 2\tfrac{\alpha_k}{1+\alpha_k} = \| \tilde \U^\dag e_k \|^2_{L^2} - s_k' \geq -s_k'. \label{eq:U_alpha_s} \end{equation}
As a result,
\[ 2\alpha_k \geq -\frac{2s_k'}{2+s_k'} \geq -s_k'. \]
Applying this to \eqref{eq:V_alpha_s} means we can extract a bound on $\bV^\dag e_k$:
\[ \| \tilde \V^\dag e_k \|^2_{L^2} \leq s_k + s_k'. \]
Similarly, applying the inequality in \eqref{eq:V_alpha_s} to \eqref{eq:U_alpha_s} gives us that 
\[  \| \tilde \U^\dag e_k \|^2_{L^2} = s_k' + 2\tfrac{\alpha_k}{1+\alpha_k} \leq s_k' + \alpha_k \leq s_k' + s_k. \]

This is nice for us because the $s_k, s_k'$ decay very quickly (from earlier propositions). Applying $\bV^\dag$ to some function $\varphi \in L^2$ gives that 
\begin{align*} \| \tilde \V^\dag \varphi\|_{L^2} &\leq \sum_{k \in \Z} \| \tilde \V^\dag e_k \|_{L^2} \F\varphi(k) \\
&\leq \left( \sum_{k\in\Z} \sigma(k)^2 \| \tilde \V^\dag e_k \|_{L^2}^2 \right)^{1/2} \left( \sum_{k\in\Z} \sigma(k)^{-2} |\F\varphi(k)|^2  \right)^{1/2}\\
&\leq \left( \sum_{k\in\Z} \sigma(k)^2 (s'_k + s_k)\right)^{1/2} \|\varphi\|_{W^{\sigma^{-1}}}\\
&\leq C \|\varphi\|_{W^{\sigma^{-1}}},
\end{align*}
with the last line coming from Lemmas~\ref{l:SDashAreZero} and~\ref{l:SDecay}.

Similarly,
\[ \| \tilde \U^\dag \varphi\|_{L^2} \leq  C \|\varphi\|_{W^{\sigma^{-1}}}. \]

By the duality of the $W^{\sigma^{-1}}$ and $W^\sigma$ norms,
\[ \| \tilde \U \psi\|_{W^\sigma}, \| \tilde \V \psi\|_{W^\sigma} \leq C \|\psi\|_{L^2} \]
for all $\psi \in L^2$, as required. 
On the other hand, 
\[\tilde \V^\dag = (I + \tilde \V^\dag)^{-1} - I = - \tilde \V^\dag (I + \tilde \V^\dag)^{-1} = -  (I + \tilde\V^\dag) \]
which gives
\[ \|\tilde \V^\dag \psi\|_{W^\sigma} \leq  \| I + \tilde\V^\dag \|_{L^2} \| \psi \|_{L^2} \leq \|\bU^\dag \M \bU\|_{L^2}^{1/2} C \| \psi \|_{L^2} \leq C \| \psi \|_{L^2}. \]
and similarly for $\tilde \U^\dag$.
\end{proof}

\begin{proof}[Proof of Theorem~\ref{t:CholeskyBound}]
For all four operators this result is a simple application of Lemma~\ref{l:GraveOperators} and Proposition~\ref{p:TildeOperators}. For example, let us consider 
\[ \U = (I + \tilde\U) \bU. \]
We have that 
\begin{align*}
\| \U \|_{W^\sigma} &\leq (1 + \|\tilde\U\|_{W^\sigma}) \| \bU\|_{W^\sigma}\\
&\leq (1 + \|\tilde\U\|_{L^2 \to W^\sigma}) \| \bU\|_{W^\sigma}\\
&\leq \left(1 + C\right) (1 + 2e^{qA}) M^{1/2},
\end{align*}
as required.
\end{proof}

The following result arises from the diagonal structure of the Dirichlet kernel in the (orthogonal) Fourier basis:
\begin{proposition}\label{p:DirichletError}
	For all $K \in \N$ and $\tau/\sigma$ decreasing, the Dirichlet kernel approximates the identity as
	\[ \|I - \D_K \|_{W^\sigma\to W^\tau} = \sup_{|k|\leq K} \frac{\tau}{\sigma} \leq \frac{\tau(K)}{\sigma(K)}. \]
\end{proposition}

We can now prove the main theorem on polynomial approximation:
\begin{proof}[Proof of Theorem~\ref{t:GalerkinGood}]
	From Proposition~\ref{p:OrthoProjectionDecomposition} we have that 
	\[I - \P_K = \U^{-1} (I - \D_K) \U = \V^\dag (I-\D_K) \U\]
	 and so using Theorem~\ref{t:CholeskyBound},
	\[ \| I - \P_K \|_{W^\sigma \to W^\tau} \leq \|\V^\dag\|_{W^\tau} \|I - \D_K\|_{W^\sigma \to W^\tau} \|U\|_{W^\tau} \leq C_\P \| I - \D_K \|_{W^\sigma \to W^\tau}, \]
	where $C_{\P} = C_{\triangle}^2$. Combining this with Proposition~\ref{p:DirichletError} gives the required result.
	as required.
 \end{proof}

 \section{Transfer \rev{and Koopman} operator results}\label{s:Transfer}
 
\rev{Recall from Section~\ref{s:Results} that we can lift $f$ onto $\R$ by $\hat f$, and this has an inverse $v = \hat f^{-1}$.} We know $f$ is $w$-to-one for some $w\geq 2$, so we expect $v$ to be $2\pi w$-periodic.
 
Define the operator $\L_\mu: \Hd^2_\zeta \circlearrowleft$ as follows:
 \begin{equation} (\L_\mu \varphi)(z) =  \sum_{j=0}^{w-1} J_\mu(z+2\pi j) \varphi(v(z+2\pi j)), \label{eq:TransferOperator}\end{equation}
 where 
 \begin{equation}J_\mu(z) = \sign f'(0) \frac{v'(z) \md(v(z))}{\md(z)}.\label{eq:JMuDef}\end{equation}
  Note that all $\L_\mu$ are conjugate to $\L_1$ \rev{(the transfer operator with respect to Lebesgue measure)} as $\L_\mu \varphi = \md^{-1} \L_1 (\md \varphi)$. 
 
\begin{theorem}\label{t:GoodL2Bound}
	Suppose $|\log f'|_{C^\alpha(\T_\zeta)} \leq D$. Then there exists $C_{\Hd^2}$ depending only on $D,\zeta,\alpha$ such that for $t \in [0,\zeta]$ and $u > \kappa t$,
 	\[ \| \L_\mu \|_{\Hd^2_u \to \Hd^2_t} \leq C_{\Hd^2} M^2. \]
\end{theorem}
This improves the bounds in the one-dimensional case of \cite[Lemma~5.3]{Bandtlow08} to be independent of $u - \kappa t$.

As part of proving this, we will prove a standard uniform bound in $\Hd^\infty$ spaces \`a la \cite{Slipantschuk20}, and, what is new, a uniform bound in $\Hd^1$ spaces:
\begin{proposition}\label{p:GoodLInfBound}
	There exists $C$ depending only on $D,\zeta,\alpha$ such that for $t \in [0,\zeta]$ and $u \geq \kappa t$,
	\[ \| \L_\mu \|_{\Hd^\infty_u \to \Hd^\infty_t} \leq C M^2. \]
\end{proposition}

\begin{proposition}\label{p:GoodL1Bound}
	There exists $C$ depending only on $D,\zeta,\alpha$ such that for $t \in [0,\zeta]$ and $u > \kappa t$,
	\[ \| \L_\mu \|_{\Hd^1_u \to \Hd^1_t} \leq C M^2 \]
\end{proposition}

The proofs of these three results are given in the Appendix.

The Perron-Frobenius operator $\L_\mu$, {\rev which due to its simple formula \eqref{eq:TransferOperator} extends naturally to $\L^2(\mu)$}, is in fact just the adjoint of the Koopman operator:
\begin{proposition}\label{p:DualL2}
	For all $\varphi,\psi \in L^2(\mu)$,
	\[ \int_\T \overline{\K \psi}\, \varphi\,\dd \mu = \int_{\T} \overline{\psi}\,\L_\mu\varphi\,\dd\mu \]
\end{proposition}
\begin{proof}
	Suppose $\psi,\varphi \in L^2$. Then
	\[ \int_\T \overline{\K \psi}\, \varphi\,\dd\mu = \int_{\T} \overline{\psi}\circ f\, \varphi\,\dd\mu. \]
	With a $w$-to-one change of variables $f(x) = y$ we find
	\[ \int_{\T} \overline{\psi}\circ f\, \varphi\,\dd\mu = \int_{\T} \sum_{j=0}^{w-1} \overline{\psi(y)} \varphi(v(y+2\pi j)) \mu(v(y+2\pi j)) |v'(y+2\pi j)|\,\dd y\]
	from which, since $\sign v' = \pm f$, the required identity follows.
\end{proof}

\begin{proposition}\label{p:DualGeneral}
	For all $t \in (0,\zeta]$ and $u \in (\kappa t,\zeta]$, $\K$ extends to a bounded operator on $\Hd^2_{-t} \to \Hd^2_{-u}$, and for all $\varphi \in \Hd^2_u, \psi \in \Hd^2_{-t}$,
	\[ \int_\T \overline{\K \psi}\, \varphi\,\dd\mu = \int_{\T} \overline{\psi}\,\L_\mu\varphi\,\dd\mu \]
\end{proposition}
\begin{proof}
	Suppose $\psi,\varphi \in L^2(\mu) = L^2(\Leb,\T)$. Then the same adjoint relationship is obeyed as in Proposition~\ref{p:DualL2}, and 
	\begin{align*} \|\K\psi\|_{\Hd^2_{-u}} &= \sup_{\|\chi\|_{\Hd^2_u} = 1} \left|\int_\T \K \psi\, \chi\,\dd x\right|\\
	&= \sup_{\|\chi\|_{\Hd^2_u} = 1}  \left|\int_{\T} \psi\, \L_1\chi\,\dd x\right| \\
	&\leq  \|\L_1\|_{\Hd^2_t \to \Hd^2_u} \|\psi\|_{\Hd^2_{-t}} \leq C_{\Hd^2} \|\psi\|_{\Hd^2_{-t}}.\end{align*}
	applying Theorem~\ref{t:GoodL2Bound} to $\L_1$. By completion, $\K: \Hd^2_{-t} \to \Hd^2_{-u}$ is bounded in norm. Since by considering the Fourier duality, $L^2(\mu) = L^2(\Leb/2\pi)$ is dense in $\Hd^2_{-t}$, the adjoint relation is preserved. 
\end{proof}


\begin{proof}[Proof of Theorem~\ref{t:Deterministic}]
	By the last proposition we have that $\K$ extends to an endomorphism of $\Hd^2_t$.
	
	We first show that $\K: \Hd^2_{-t} \circlearrowleft$ is compact. For all $u \in [0,t)$, inclusion $I: \Hd^2_t \to \Hd^2_u$ is compact since from Proposition~\ref{p:DirichletError}, $\{\D_K\}_{K\in\N}$ are a family of finite-rank operators $\Hd^2_t \to \Hd^2_u$ with $\D_K \to I$ in operator norm. By duality, inclusion $I: \Hd^2_{-u} \to \Hd^2_{-t}$ is also compact. Then for any choice of $u \in (\kappa t,t)$, $\K: \Hd^2_{-t} \to \Hd^2_{-t}$ is the composition of the bounded operator $\K: \Hd^2_{-u} \to \Hd^2_{-t}$ with this compact inclusion, giving compactness of $\K$ in $\Hd^2_{-t}$.
	
	Let us consider $\K_K = \P_K \K$. From Theorem~\ref{t:GalerkinGood} and the fact that $\sigma_u(K) / \sigma_t(K) \leq 2 e^{-K(t-u)}$, we have that
	\[ \| \P_K - I \|_{\Hd^2_t \to \Hd^2_u} \leq 2C_{\P}e^{-K(t-u)}, \]
	but we need to convert this into the dual norm. Because $\P_K$ is self-adjoint in $L^2(\mu)$, we have for $\psi,\chi \in L^2(\Leb,\T)$ (i.e. in $L^2(\mu)$) that
	\[ \int \P_K \psi\, \chi\,\dd x = \int \psi\, \md\P_K(\md^{-1}\chi)\,\dd x \]
	and so, by the same argument as in Proposition~\ref{p:DualGeneral}, $P_K$ and $I - \P_K$ extend to bounded operators $\Hd^2_{-u} \to \Hd^2_{-t}$, with
	\[ \| I - \P_K \|_{\Hd^2_{-u} \to \Hd^2_{-t}} \leq \|\md\|_{\Hd^\infty_t} \|I - \P_K\|_{\Hd^2_t \to \Hd^2_u} \|\md\|_{\Hd^\infty_u} \leq M^2 \cdot 2C_{\P}e^{-K(t-u)} \]
	
	Consequently, using Proposition~\ref{p:DualGeneral} and Theorem~\ref{t:GoodL2Bound} we obtain that for all $u > \kappa t$,
	\begin{align}\label{eq:KProjDistance}
	\| \K - \P_K \K \|_{\Hd^2_{-t}} &\leq \| I - \P_K \|_{\Hd^2_{-u} \to \Hd^2_{-t}} \| \K \|_{\Hd^2_{-t} \to \Hd^2_{-u}} \\ &\leq 2M^2 C_{\Hd^2} C_{\P} e^{-K(t-u)}.
	\end{align}
	Taking the infimum over $u$ we get the required result for the first term being bounded. 
	
	Similarly, for all $v < t$, we have
		\begin{align}\label{eq:KProjDistance2}
	\| \K - \K\P_K \|_{\Hd^2_{-\kappa t}} &\leq \| \K \|_{\Hd^2_{-v} \to \Hd^2_{-\kappa t}}  \| I - \P_K \|_{\Hd^2_{-\kappa t} \to \Hd^2_{-v}} \\ &\leq 2M^2 C_{\Hd^2} C_{\P} e^{-K(v - \kappa t)}
	\end{align}
	and can take an infimum over $v$ for the second term.
\end{proof}

\begin{proof}[Proof of Corollary~\ref{c:Spectrum}]
%
	
The dual of $\K: \Hd^2_{-t} \circlearrowleft$ is $\L_\mu: \Hd^2_t \circlearrowleft$, {\rev from Proposition~\ref{p:DualGeneral}}, \rev{which therefore inherits $\K$'s compactness from Theorem~\ref{t:Deterministic}}. This gives us that
$$ \int \varphi\, \psi\circ f^n\,\dd\mu = \int \L_\mu^n \varphi\, \psi\,\dd \mu = \sum_{j=1}^J \int_\T \L_\mu^n \Pi_j \varphi\, \psi\,\dd\mu + o(\lambda_J^n)$$
as $n\to \infty$, where $\Pi_j$ are the spectral projections onto the $\lambda_j$-generalised eigenspace, \rev{i.e., $\ker (\L_\mu - \lambda_j I)^{M_j}$ where $M_j$ is the algebraic multiplicity of $\lambda_j$}. This means that $\int_\T \L_\mu^n \Pi_j \varphi\, \psi\,\dd\mu = \mathcal{O}(n^{M_j} \lambda_j^n)$. Note that since $\sigma(\L_\mu) = \sigma(\L_1)$, these are the same eigenvalues we would get if we were considering correlations against Lebesgue measure or the physical measure of $f$, so they are the Ruelle-Pollicott resonances in the usual sense.

	\rev{We now prove the convergence result. We know that the right (generalised) eigenfunctions of  $\K_K = \P_K \K\P_K$ and $\P_K\K$ must both lie in the image of $\P_K$. Hence, since the action of $\P_K$ on its image is the identity and this image is finite-dimensional, we know $\K_K$ and $\P_K \K$ must therefore have the same generalised right eigenspaces, and the same spectrum.  (Their left eigenfunctions are actually also matched in as much as $a$ is a left eigenfunction of $\P_K\K$ if and only if $\P_K a$ is a left eigenfunction of $\K_K$.) Similarly, $\K_K$ and $\K\P_K$ agree in spectrum and left eigenfunctions.}
	
	\rev{We can then prove convergence of right eigenfunctions of $\K_K$ to $\K$ by applying \cite[Theorem~1]{Osborn75} to $\P_K \K$ converging to $\K = (\L_\mu)^*$. Similarly, we can get convergence of left eigenfunctions by applying it to $(\K \P_K)^*$ converging to to $\K^* = \L_\mu$. Convergence of spectrum we get from \cite[Theorem 6]{Osborn75}, noting that it doesn't matter which value of $t$ we use to estimate the eigenvalues, and setting $t=\zeta$ therefore gives the optimal bound.}
\end{proof}

\begin{proof}[Proof of Theorem~\ref{t:SimpleSpectral}]
	The projected operators $\{\P_{L}\K\}_{L \in \N}$ are rank $2L-1$ approximations of the full operator $\K$, and using \eqref{eq:KProjDistance},
		$$ \| \P_L \K - \K \|_{\Hd^2_{-t}} \leq 2M^2 C_{\Hd^2} C_{\P} e^{-L(t-u)}. $$
		Similarly, $\{\P_{L} \K\}_{L \leq K}$ are rank-$2L-1$ approximations of another projected operator $\K_K$, and by comparing both against the full operator, we get that
	$$ \| \P_L \K - \K_K \|_{\Hd^2_{-t}} \leq \| \P_L \K - \K \|_{\Hd^2_{-t}} + \| \K_K - \K \|_{\Hd^2_{-t}} \leq 4M^2 C_{\Hd^2} C_{\P} e^{-L(t-u)} $$ 
	for $u \in [\kappa t,t]$,
	so taking infima,
	$$ \| \P_L \K - \K_K \|_{\Hd^2_{-t}},\| \P_L \K - \K \|_{\Hd^2_{-t}} \leq 4M^2 C_{\Hd^2} C_{\P} e^{-L(1-\kappa)t}.$$
	This means the $2L-1$th and $2L$th approximation numbers (and therefore corresponding singular values) of $\P_K \K, \K$ are bounded by the above constant. \rev{\cite[Corollary~5.3]{Bandtlow15} provides uniform bounds on the Hausdorff distance of the spectra of two operators given their approximation numbers, and gives us that} we have that the Hausdorff distance between $\sigma(\P_K\K), \sigma(\K)$ are bounded by $Ce^{-c\sqrt{K}}$ for some $C,c>0$.
	
	To prove the rest of the theorem, we apply Corollary~\ref{c:Spectrum} when the $\Pi_j$ are rank-one. That is, recalling that $\beta_j(\psi) := \int \psi b_j\,\dd x$, we can write $\Pi_j = \rev{a}_j \beta_j$ with $\beta_j \rev{a}_j = 1$, and therefore 
	\begin{align*}\L_\mu^n \rev{a}_j &= \lambda_j^n \rev{a}_j\\
	\rev{\beta_j \L_\mu^n} &= \rev{\lambda_j^n \beta_j}, \end{align*}
	\rev{which by duality gives that $a_j$ and $\beta_j$ are respectively left and right eigenfunctions of $\K$.} Note that the constant $h = 1 - \kappa$.
\end{proof}

\appendix

\section{Proofs of transfer operator bounds}

In this section, we prove various tighter bounds on the transfer operator, specifically in $\Hd^p_u \to \Hd^p_t$ for $p \in \{1,2,\infty\}$. We will actually prove results for an extension of the transfer operator in these spaces to more general spaces of harmonic functions. For $t>0$ and $p \in [1,\infty]$ us define the spaces
\begin{equation*}
	\Ad^p_t = \{ \varphi: \overline{\T_\zeta} \to \C : (\tfrac{\partial^2}{\partial x^2} + \tfrac{\partial^2}{\partial y^2}) \varphi(x+iy) = 0; \| \varphi|_{\partial \T_\zeta}\|_{L^p} < \infty \}
\end{equation*}
with the norm
\begin{equation*}
	\| \varphi \|_{\Ad^p_t} =  \| \varphi|_{\partial \T_\zeta}\|_{L^p} = \begin{cases} \frac{1}{4\pi} \int_\T |\varphi(x+it)|^p + | \varphi(x-it)|^p\, \dd x, & p < \infty \\ \sup_{z\in\partial\T_\zeta} |\varphi(z)|. \end{cases}
\end{equation*}
Harmonicity means that a function in $\Ad^p_t$ is uniquely determined by its values on the boundary: in particular from any function on $\overline{\T_\zeta}$ with an $L^p$ restriction to $\partial \T_\zeta$, we can construct a harmonic function $\mathcal{I}_t\varphi \in \Ad^p_t$ that matches $\varphi$ on the boundary using a kernel operator
\begin{equation}\label{eq:IKernelDef}
\mathcal{I}_t\varphi(x+iy) = \int_\T (k_t(\theta + it - z) \varphi(\theta + it) + k_t(\theta + it - z) \varphi(\theta - it))\,\dd \theta
\end{equation}
where
\begin{equation} \label{eq:HarmonicKernelDef}
	k_t(x+iy) = \sum_{n\in\Z} \frac{\pi}{4t} \frac{2 \cot \tfrac{\pi y}{4t} \sech^2 \tfrac{\pi (x+2n)}{4t}}{ \cot^2 \tfrac{\pi y}{4t} + \tanh^2 \tfrac{\pi (x+2n)}{4t}} \geq 0.
\end{equation}
Note that functions in $\Ad^p_t$ are mapped to themselves under $\mathcal{I}_t$.

Since all holomorphic functions are harmonic, the Hardy space $\Hd^p_t$ is a closed subset of $\Ad^p_t$ with the same norm. We can extend the domain of $\L_\mu$ to $\Ad^p_t$ in the obvious way:
\[ (\L_\mu \varphi)(z) = \sum_{j=0}^{w-1} J_\mu(z+2\pi j) \varphi(v(z+2\pi j)) \] 
but the output functions are not necessarily harmonic, i.e. in $\Ad^p_t$. Instead, we will study the following extension of $\L_\mu$, which is closed under harmonic functions:
\begin{equation*}
	\tilde{\L}_{t,\mu} \varphi = \mathcal{I}_t \L_\mu \varphi = \mathcal{I}_t \left[  \right],
\end{equation*}
and it is this operator we will prove is bounded on $\Ad^p_t \to \Ad^p_u$ for the different $p$. Note that for any $\varphi \in \Ad^p_t$,
	\begin{equation} \| \tilde{\L}_{t,\mu} \varphi \|_{\Ad^p_t} = \|(\mathcal{I}_t\L_\mu \varphi)|_{\partial \T_\zeta}\|_{L^p}  = \|(\L_\mu \varphi)|_{\partial \T_\zeta}\|_{L^p},  \label{eq:BoundaryPrinciple}\end{equation}
even though $\L\mu_\varphi$ may not even be harmonic, simply because $\mathcal{I}_t$ matches functions on the boundary.

\begin{proof}[Proof of Proposition~\ref{p:GoodLInfBound}]
	We in fact prove this bound for $\tilde\L_{t,\mu}: \Ad^\infty_u \to \Ad^\infty_t$.
	
	From \eqref{eq:BoundaryPrinciple} we have for $\varphi \in \Ad^p_u$ that 
		\begin{align*} \| \tilde{\L}_{t,\mu} \varphi \|_{\Ad^\infty_t} &= \sup_{z \in \T_t} |(\L_\mu \varphi)(z)|\\
	&= \sup_{z \in \T_t} \sum_{j=0}^{w-1} J_\mu(z+2\pi j) \varphi(v(z+2\pi j).  \end{align*}
	Now, from its definition in \eqref{eq:JMuDef}, $|J_\mu(z + 2\pi j)| \leq M^2 | v'(z+2\pi j)| $. Standard application of our H\"older distortion assumption on $v$ implies that 
	\begin{align*} |v'(z+2\pi j)| 
		& \leq e^{D(\pi^2+\zeta^2)^{\alpha/2}} \tfrac{\Leb v([\Re z - \pi + 2\pi j, \Re z + \pi + 2\pi j])}{2\pi}\end{align*}
	so 
	$$\sum_{j=0}^{w-1} |J_\mu(z+2\pi j)| \leq M^2 e^{D(\pi^2+\zeta^2)^{\alpha/2}}. $$
	
	 Furthermore, $\Im v(z + 2\pi j) \leq |v(z + 2\pi j) - v(\Re z + 2\pi j)| \leq \kappa^{-1} \Im z \leq \kappa^{-1} t$, so $v(z + 2\pi j) \in \T_{\kappa t} \subset \T_u$, so by the maximum principle
	\[ \| \tilde{\L}_{t,\mu} \varphi \|_{\Ad^\infty_t} \leq e^{D(\pi+\zeta)^{\alpha}} 2\pi M^2 \sup_{z \in \T_u} |\varphi(z)| \leq e^{D(\pi^2+\zeta^2)^{\alpha/2}} 2\pi M^2 \|\varphi\|_{\Ad^\infty_u} \]
	as required.
\end{proof}

\begin{proof}[Proof of Proposition~\ref{p:GoodL1Bound}]
		We in fact prove this bound for $\tilde\L_{t,\mu}: \Ad^1_u \to \Ad^1_t$.
	Note too that we only need consider $u$ smaller than $\zeta$: larger $u$ follows by inclusion.

	Using \eqref{eq:BoundaryPrinciple}, and the fact that the harmonic functions in $\Ad^1_u$ obey $\varphi = \mathcal{I}_u\varphi$, we have that
	\[ \| \tilde{\L}_{t,\mu} \varphi \|_{\Ad^1_u}  \leq \sum_{\pm_a, \pm_b} \sum_{j=0}^{w-1} \int_\T \int_\T |J_\mu'(x+2\pi j \pm_b it)| k_u(\theta \pm_a iu - v(x\pm_b +2\pi j + it)) |\varphi(\theta + iu)|\,\dd \theta\,\dd\omega. \]
	This bound of $\tilde \L_{t,\mu}$ by a kernel implies that
	\begin{align} \| \tilde \L_{t,\mu} \|_{\Ad^1_t \to \Ad^1_u} &\leq \sup_{\theta \in \T, \pm_a} \sum_{\pm_b} \sum_{j=0}^{w-1} \int_\T |J_\mu'(x+2\pi j \pm_b it)| k_u(\theta \pm_a iu - v(x\pm_b +2\pi j + it))\,\dd\omega \notag \\
	&\leq M^2  \sup_{\theta \in \T} \sum_{j=0}^{w-1} \sum_{\pm} \int_\T |v'(x + 2\pi j + it)| k_u(\theta \pm iu - v(x +2\pi j + it))\,\dd\omega \notag \\
	&\leq M^2 \sup_{\theta \in \T} \sum_{\pm} \int_{v(w\T + it)} k_u(\theta \pm iu - z)\, |\dd z| \label{eq:LineIntegral}
	\end{align}
	where we used the symmetry of $v$ under conjugation and bounds on $\md$, followed by a change of variables.
	
	On $\T_u$ we can bound
	\begin{equation} k_u(x + iy) \leq C_u + \frac{2 y}{x^2 + y^2}. \label{eq:KernelBound} \end{equation}
	for some $C_u$ increasing in $u$. Note that $k_u$ blows up near $0$ and is not integrable along certain curves (e.g. $y = |x|$, which maximises $k_u$ for fixed $x$), but has constant integral along lines of constant $y$. We will show that when our curve $v(w\T + it)$ is close enough to $x+iy = 0$, it must be locally close to a line of constant $y$.
	
	If $v'>0$ on $\T$, then $|\arg v'(x+it)| = \Im \log v'(x+it) \leq D\zeta^\alpha \leq \tfrac{\pi}{3}$, so $|v'(x+it)| \leq 2|\Re v'(x+it)|$, we have that $v(w\T + it)$ on the complex plane identified with $\R^2$ can be written as a graph of a function $V: \T \to [-\kappa t,\kappa t]$ with $|V'| \leq \sqrt{3}$. A similar thing holds when $v' < 0$ on $\T$. This means that
	\begin{equation}\label{eq:KernelGraphIntegral}
		 \int_{v(w\T + it)} k_u(\theta \pm iu - z) = \int_{\T} k_u(\theta \pm iu - (x + iV(x))) \sqrt{1 + V(x)^2}\,\dd x \leq 2\int_{\T} k_u(\theta \pm iu - (x + iV(x)))\,\dd x 
	\end{equation}
	
	Let us consider separately two parts of our curve: the set $X = \{x \in [\theta-2u,\theta+2u] : |V(x)| > u/2\}$, and the remainder $\T \backslash X$.  Using \eqref{eq:KernelBound} on $\T \backslash X$, we can bound
	\[ k_u(\theta \pm iu - (x+iV(x))) \leq C_u + \begin{cases} \frac{u}{(x-\theta)^2 + (u/2)^2},& |x-\theta| < u/2 \\ \frac{1}{x-\theta}, & |x-\theta| \in [u/2,2u]\\ \frac{4u}{(x-\theta)^2+4u^2}, & |x-\theta| > 2u. \end{cases}\]
	 This means that the $\T \backslash X$ contribution to \eqref{eq:KernelGraphIntegral} gives
	\begin{equation*}\label{eq:TNotXPart}
		 \int_{\T\backslash X} k_u(\theta \pm iu - (x + iV(x)))\,\dd x \leq |\T\backslash X| C_u + \pi + \log 4. 
	\end{equation*}

	On the other hand for $x \in X$,
	\[ u \leq \Im v(x+it) \leq \int_0^t |v'(x+i\tau)|\,\dd\tau \leq t e^{D \zeta^\alpha} |v'(x+it)| \]
	by H\"older distortion on $\T_\zeta$, so for $x \in X$,
	\[ |v'(x+it)|^{-1} \leq 2t e^{D \zeta^\alpha}/u \leq 2\zeta e^{D\zeta^\alpha}/u.  \]
	As a result, the direction of the curve $v(wX + it)$ varies as an $\alpha$-H\"older function of arc-length with constant $D (2\kappa^{-1} e^{D \zeta^\alpha})^\alpha$, and so $V' = \tan \arg v'$ has an $\alpha$-H\"older constant $4 D (2\kappa^{-1} e^{D \zeta^\alpha}/u)^\alpha =: K^\alpha u^{-\alpha}$. 
	
	Now let us consider the curve $x+iV(x)$ for $x$ close to $\theta$. Since $V$ must lie in $[-u,u]$, it must reach a maximum (resp. minimum) in $[-u,u]$, and so
	\begin{equation}
	|V'(\theta)| \leq K' \left( \tfrac{u-V(\theta)}{u}\right)^{\alpha/(\alpha+1)} \label{eq:CurveGradientBound}
	\end{equation} 
	with $K' = K \tfrac{\alpha}{\alpha+1}$. Taylor approximation gives us that 
	\begin{equation} | \tfrac{V(x)- V(\theta)}{u} - V'(\theta) \tfrac{x - \theta}{u} | \leq K''  |\tfrac{x - \theta}{u}|^{1+\alpha}. \label{eq:HolderDirectionBound} \end{equation}
	with $K'' = \frac{1}{1+\alpha} K^\alpha$. Combining the previous two equations we get that
	\begin{equation}
			K' \left(\tfrac{u-V(\theta)}{u}\right)^{\tfrac{\alpha}{\alpha+1}} (x-\theta) - K'' |\tfrac{x - \theta}{u}|^{1+\alpha} \leq \tfrac{V(x) - V(\theta)}{u} \leq K' \left(\tfrac{u-V(\theta)}{u}\right)^{\tfrac{\alpha}{\alpha+1}} (x-\theta) + K'' |\tfrac{x - \theta}{u}|^{1+\alpha}
	\end{equation} 
Noting the scaling of this equation with $u$, the fact that the length of $X$ must be less than $2u$, and the bound for $k_u$ in \eqref{eq:KernelBound}, we find there exists a uniform bound 
\[ \int_{\T\backslash X} k_u(\theta \pm iu - (x + iV(x)))\,\dd x \leq |X| C_u + C'  \]
for $C'$ independent of $u$. Substituting this and \eqref{eq:TNotXPart} into \eqref{eq:KernelGraphIntegral} we get
\[   \int_{v(w\T + it)} k_u(\theta \pm iu - z)  \leq 2\pi C_u + \pi + \log 4 + C' \leq 2\pi C_{\zeta} + \pi + \log 4 = C''. \]
Substituting this into \eqref{eq:LineIntegral} gives the required uniform bound.
%
\end{proof}

We will need the following result to relate different $\Ad^p$ spaces with each other, allowing us to prove Theorem~\ref{t:GoodL2Bound}:
\begin{lemma}
	For $t \in [0,\zeta]$, $u > \kappa t$, and $p,p' \in [1,\infty]$, $\tilde{\L}_{t,\mu}: \Ad^p_u \to \Ad^{p'}_t$ is bounded.
\end{lemma}
\begin{proof}
	Proposition~\ref{p:GoodLInfBound} tells us that $\tilde\L_{t,\mu}$ is bounded $\Ad^\infty_{\kappa t} \to \Ad^\infty_t$. We then need to resolve the integrability parameters.
	
	The boundedness of the kernel $k_u(x+iy)$ for $y > 0$ via \eqref{eq:KernelBound} means that, by the definition of  \eqref{eq:IKernelDef}, $\mathcal{I}_u$ is bounded $\Ad^p_u \to \Ad^{\infty}_{\kappa t}$, and therefore so is function inclusion. On the other hand, $L^p$ inclusions on the boundary gives a bounded inclusion from $\Ad^{p'}_t \to \Ad^\infty_t$.
\end{proof}

\begin{proof}[Proof of Theorem~\ref{t:GoodL2Bound}]
	The spaces $\Ad^2_t$ are Hilbert spaces with the $L^2(\partial \T_\zeta)$ inner product.
	As a result, for any $t \in [0,\zeta]$, $u > \kappa t$ there exists an operator $\mathcal{J}_{\mu}^{u,t}: \Ad^2_t \to \Ad^2_u$ adjoint to $\tilde{\L}_{t,\mu}: \Ad^2_u \to \Ad^2_t$. We can then say that 
	\[ \| \tilde{\L}_{t,\mu} \|_{\Ad^2_u \to \Ad^2_t}^2 = \| \tilde{\L}_{t,\mu} \mathcal{J}^{u,t}_\mu \|_{\Ad^2_u}, \]
	and in fact by self-adjointness of $\tilde{\L}_{t,\mu} \mathcal{J}^{u,t}_\mu$ that
	\[ \| \tilde{\L}_{t,\mu} \|_{\Ad^2_u \to \Ad^2_t}^2 = \sqrt[n]{\| (\tilde{\L}_{t,\mu} \mathcal{J}^{u,t}_\mu)^n \|_{\Ad^2_u}}. \]
	Now,
	\[ \| (\tilde{\L}_{t,\mu} \mathcal{J}^{u,t}_\mu)^n \|_{\Ad^2_u} \leq \|\tilde{\L}_{t,\mu}\|_{\Ad^1_t \to \Ad^2_u} \| (\mathcal{J}^{u,t}_\mu \tilde{\L}_{t,\mu}) \|_{\Ad^1_t}^{n-1} \| \mathcal{J}^{u,t}_\mu \|_{\Ad^2_u \to \Ad^1_t}. \]
	Since $\Ad^2_t$ has a bounded inclusion into $\Ad^1_t$, $ \mathcal{J}^{u,t}_\mu$ is bounded $\Ad^2_u \to \Ad^1_t$; by the previous lemma, $\tilde{\L}_{t,\mu}$ is bounded $\Ad^1_t \to \Ad^2_u$. All the operators in the above expression are therefore bounded, and in particular for some $C$,
	\begin{equation}
		\| \tilde{\L}_{t,\mu} \|_{\Ad^2_u \to \Ad^2_t}^2 \leq \inf_{n\geq 1} \sqrt[n]{C \| ( \mathcal{J}^{u,t}_\mu \tilde{\L}_{t,\mu}) \|_{\Ad^1_t}^{n-1}} =  \| \mathcal{J}^{u,t}_\mu \tilde{\L}_{t,\mu} \|_{\Ad^1_t}. \label{eq:L2IntoL1}
	\end{equation}
	
	Suppose then that $\varphi \in \Ad^2_t$, and so $\mathcal{J}^{u,t}_\mu \tilde{\L}_{t,\mu} \varphi$ is as well. Defining $\omega = \mathcal{I}_t\left[\tfrac{\mathcal{J}^{u,t}_\mu \tilde{\L}_{t,\mu} \varphi}{|\mathcal{J}^{u,t}_\mu \tilde{\L}_{t,\mu} \varphi|}\right] \in \Ad^\infty_\mu$ we have that almost everywhere on the boundary of $\T_t$,
	\[ \bar{\omega} \mathcal{J}^{u,t}_\mu \tilde{\L}_{t,\mu} \varphi = |\mathcal{J}^{u,t}_\mu \tilde{\L}_{t,\mu} \varphi|. \]
	
	We then have
	\begin{align*} \| \mathcal{J}^{u,t}_\mu \tilde{\L}_{t,\mu} \varphi \|_{\Ad^1_t} &= \frac{1}{4\pi} \int_\T (\bar\omega \mathcal{J}^{u,t}_\mu \tilde{\L}_{t,\mu} \varphi)(\theta + i\zeta) + (\bar\omega \mathcal{J}^{u,t}_\mu \tilde{\L}_{t,\mu} \varphi)(\theta - i\zeta) \,\dd\theta\\
	& = \langle \omega, \mathcal{J}^{u,t}_\mu \tilde{\L}_{t,\mu} \varphi \rangle_{\Ad^2_t} = \langle \tilde{\L}_{t,\mu} \omega, \tilde{\L}_{t,\mu} \varphi \rangle_{\Ad^2_t} \\ 
	& =\frac{1}{4\pi} \int_\T (\mathcal{L}_\mu \omega\, \tilde{\L}_{t,\mu} \varphi)(\theta + i\zeta) + (\mathcal{L}_\mu \omega\, \tilde{\L}_{t,\mu} \varphi)(\theta - i\zeta) \,\dd\theta\\
	& \leq \| \mathcal{L}_\mu \|_{\Ad^\infty_t} \frac{1}{4\pi} \int_\T (\tilde{\L}_{t,\mu} \varphi)(\theta + i\zeta) + (\tilde{\L}_{t,\mu} \varphi)(\theta - i\zeta) \,\dd\theta\\
	& \leq \| \mathcal{L}_\mu \|_{\Ad^\infty_t} \| \mathcal{L}_\mu \|_{\Ad^1_t} \| \varphi\|_{\Ad^1_t} 
	\end{align*}
	Applying Propositions~\ref{p:GoodLInfBound} and~\ref{p:GoodL1Bound} for $\Ad$ spaces, and substituting into \eqref{eq:L2IntoL1}, gives us what we want when $\varphi \in \Ad^2_t$. Since $\Ad^2_t$ is dense in $\Ad^1_t$, we obtain the full result by interpolation. 
	
	We can of course then go back and restrict to looking at $\L_\mu$ on $\Hd^2_u \to \Hd^2_t$.
\end{proof}

%
%

\subsection*{Acknowledgements}

The author thanks Matthew Colbrook for inspiring discussion and his comments on the manuscript.

This research has been supported by the European Research Council (ERC) under the European Union's Horizon 2020 research and innovation programme (grant agreement No 787304).

\bibliographystyle{plain}
\bibliography{edmd}

\begin{thebibliography}{10}

\bibitem{Aronszajn50}
Nachman Aronszajn.
\newblock Theory of reproducing kernels.
\newblock {\em Transactions of the American mathematical society},
  68(3):337--404, 1950.

\bibitem{Baladi00}
Viviane Baladi.
\newblock {\em Positive transfer operators and decay of correlations},
  volume~16.
\newblock World scientific, 2000.

\bibitem{Baladi18}
Viviane Baladi.
\newblock {\em Dynamical zeta functions and dynamical determinants for
  hyperbolic maps}.
\newblock Springer, 2018.

\bibitem{Bandtlow15}
Oscar~F Bandtlow and Ay{\c{s}}e G{\"u}ven.
\newblock Explicit upper bounds for the spectral distance of two trace class
  operators.
\newblock {\em Linear Algebra and its Applications}, 466:329--342, 2015.

\bibitem{Bandtlow08}
Oscar~F Bandtlow and Oliver Jenkinson.
\newblock Explicit eigenvalue estimates for transfer operators acting on spaces
  of holomorphic functions.
\newblock {\em Advances in Mathematics}, 218(3):902--925, 2008.

\bibitem{Bandtlow17}
Oscar~F Bandtlow, Wolfram Just, and Julia Slipantschuk.
\newblock Spectral structure of transfer operators for expanding circle maps.
\newblock In {\em Annales de l'Institut Henri Poincar{\'e} C, Analyse non
  lin{\'e}aire}, volume~34, pages 31--43. Elsevier, 2017.

\bibitem{Bandtlow23}
Oscar~F Bandtlow, Wolfram Just, and Julia Slipantschuk.
\newblock {EDMD} for expanding circle maps and their complex perturbations.
\newblock {\em arXiv preprint arXiv:2308.01467}, 2023.

\bibitem{Bandtlow20}
Oscar~F Bandtlow and Julia Slipantschuk.
\newblock Lagrange approximation of transfer operators associated with
  holomorphic data.
\newblock {\em arXiv preprint arXiv:2004.03534}, 2020.

\bibitem{Blumenthal17}
Alex Blumenthal, Jinxin Xue, and Lai-Sang Young.
\newblock Lyapunov exponents for random perturbations of some area-preserving
  maps including the standard map.
\newblock {\em Annals of Mathematics}, 185(1):285--310, 2017.

\bibitem{Brock18}
William~A Brock.
\newblock Nonlinearity and complex dynamics in economics and finance.
\newblock In {\em The economy as an evolving complex system}, pages 77--97. CRC
  Press, 2018.

\bibitem{Budivsic12}
Marko Budi{\v{s}}i{\'c}, Ryan Mohr, and Igor Mezi{\'c}.
\newblock Applied {K}oopmanism.
\newblock {\em Chaos: An Interdisciplinary Journal of Nonlinear Science},
  22(4):047510, 2012.

\bibitem{Chernov95}
Nikolai~I Chernov.
\newblock Limit theorems and {M}arkov approximations for chaotic dynamical
  systems.
\newblock {\em Probability Theory and Related Fields}, 101:321--362, 1995.

\bibitem{Chui82}
Charles~K Chui, J~D Ward, and P~W Smith.
\newblock {C}holesky factorization of positive definite bi-infinite matrices.
\newblock {\em Numerical Functional Analysis and Optimization}, 5(1):1--20,
  1982.

\bibitem{Colbrook22}
Matthew~J Colbrook.
\newblock The {mpEDMD} algorithm for data-driven computations of
  measure-preserving dynamical systems.
\newblock {\em arXiv preprint arXiv:2209.02244}, 2022.

\bibitem{Colbrook23}
Matthew~J Colbrook, Lorna~J Ayton, and M{\'a}t{\'e} Sz{\H{o}}ke.
\newblock Residual dynamic mode decomposition: robust and verified
  {K}oopmanism.
\newblock {\em Journal of Fluid Mechanics}, 955:A21, 2023.

\bibitem{Colbrook21}
Matthew~J Colbrook and Alex Townsend.
\newblock Rigorous data-driven computation of spectral properties of {K}oopman
  operators for dynamical systems.
\newblock {\em arXiv preprint arXiv:2111.14889}, 2021.

\bibitem{Crimmins20}
Harry Crimmins and Gary Froyland.
\newblock Fourier approximation of the statistical properties of {A}nosov maps
  on tori.
\newblock {\em Nonlinearity}, 33(11):6244, 2020.

\bibitem{Dellnitz02}
Michael Dellnitz and Oliver Junge.
\newblock {\em Set Oriented Numerical Methods for Dynamical Systems}, page 221.
\newblock Gulf Professional Publishing, 2002.

\bibitem{Froyland03}
Gary Froyland and Michael Dellnitz.
\newblock Detecting and locating near-optimal almost-invariant sets and cycles.
\newblock {\em SIAM Journal on Scientific Computing}, 24(6):1839--1863, 2003.

\bibitem{Froyland14}
Gary Froyland, Cecilia Gonz{\'a}lez-Tokman, and Anthony Quas.
\newblock Detecting isolated spectrum of transfer and {K}oopman operators with
  {F}ourier analytic tools.
\newblock {\em Journal of Computational Dynamics}, 1(2):249, 2014.

\bibitem{Fuchs14}
Armin Fuchs.
\newblock {\em Nonlinear dynamics in complex systems}.
\newblock Springer, 2014.

\bibitem{Geronimo06}
Jeffrey~S Geronimo and Andrei Mart{\'\i}nez-Finkelshtein.
\newblock On extensions of a theorem of {B}axter.
\newblock {\em Journal of Approximation Theory}, 139(1-2):214--222, 2006.

\bibitem{Hennion01}
Hubert Hennion and Lo{\"\i}c Herv{\'e}.
\newblock {\em Limit theorems for Markov chains and stochastic properties of
  dynamical systems by quasi-compactness}.
\newblock Springer, 2001.

\bibitem{Keller89}
Gerhard Keller.
\newblock Markov extensions, zeta functions, and {F}redholm theory for
  piecewise invertible dynamical systems.
\newblock {\em Transactions of the American Mathematical Society},
  314(2):433--497, 1989.

\bibitem{Keller99}
Gerhard Keller and Carlangelo Liverani.
\newblock Stability of the spectrum for transfer operators.
\newblock {\em Annali della Scuola Normale Superiore di Pisa-Classe di
  Scienze}, 28(1):141--152, 1999.

\bibitem{Klus16}
Stefan Klus and Christof Sch{\"u}tte.
\newblock Towards tensor-based methods for the numerical approximation of the
  {P}erron--{F}robenius and {K}oopman operator.
\newblock {\em Journal of Computational Dynamics}, 3(2):139--161, 2016.

\bibitem{Krall36}
HL~Krall.
\newblock On derivatives of orthogonal polynomials.
\newblock {\em Bulletin of the American Mathematical Society}, 42(6):423--428,
  1936.

\bibitem{Mauroy20}
Alexandre Mauroy, Y~Susuki, and I~Mezi{\'c}.
\newblock {\em Koopman operator in systems and control}.
\newblock Springer, 2020.

\bibitem{Osborn75}
John~E Osborn.
\newblock Spectral approximation for compact operators.
\newblock {\em Mathematics of computation}, 29(131):712--725, 1975.

\bibitem{Ruelle09}
David Ruelle.
\newblock Structure and f-dependence of the a.c.i.m. for a unimodal map f of
  {M}isiurewicz type.
\newblock {\em Communications in Mathematical Physics}, 287(3):1039--1070,
  2009.

\bibitem{Simon05}
Barry Simon.
\newblock {\em Orthogonal polynomials on the unit circle}.
\newblock American Mathematical Soc., 2005.

\bibitem{Slipantschuk20}
Julia Slipantschuk, Oscar~F Bandtlow, and Wolfram Just.
\newblock Dynamic mode decomposition for analytic maps.
\newblock {\em Communications in Nonlinear Science and Numerical Simulation},
  84:105179, 2020.

\bibitem{Son22}
Sang~Hwan Son, Hyun-Kyu Choi, Jiyoung Moon, and Joseph Sang-Il Kwon.
\newblock Hybrid {K}oopman model predictive control of nonlinear systems using
  multiple {EDMD} models: An application to a batch pulp digester with feed
  fluctuation.
\newblock {\em Control Engineering Practice}, 118:104956, 2022.

\bibitem{Williams15}
Matthew~O Williams, Ioannis~G Kevrekidis, and Clarence~W Rowley.
\newblock A data--driven approximation of the {K}oopman operator: Extending
  dynamic mode decomposition.
\newblock {\em Journal of Nonlinear Science}, 25:1307--1346, 2015.

\bibitem{Wormell19}
Caroline Wormell.
\newblock Spectral {G}alerkin methods for transfer operators in uniformly
  expanding dynamics.
\newblock {\em Numerische Mathematik}, 142(2):421--463, 2019.

\end{thebibliography}

\end{document}